\documentclass[11pt]{amsart}
\usepackage[utf8]{inputenc}
\usepackage[T1]{fontenc}
\usepackage{geometry}
\usepackage{pdflscape}
\usepackage{amsmath}
\usepackage{amssymb}
\usepackage{relsize}
\usepackage{float}
\usepackage{dsfont}
\usepackage[all,cmtip]{xy}
\usepackage[usenames,dvipsnames,svgnames,table]{xcolor}
\usepackage{array}
\usepackage{colortbl}
\usepackage{enumitem}
\usepackage{varioref}
\usepackage{xr}
\usepackage[bookmarksnumbered,colorlinks,plainpages,backref]{hyperref}
\numberwithin{equation}{section}
\newtheorem{theorem}{Theorem}[section]
\newtheorem{proposition}[theorem]{Proposition}
\newtheorem{lemma}[theorem]{Lemma}
\newtheorem{corollary}[theorem]{Corollary}
\newtheorem{conjecture}[theorem]{Conjecture}

\theoremstyle{plain}
\theoremstyle{remark}
\newtheorem{remark}[theorem]{Remark}

\renewcommand{\AA}{\mathbb{A}}
\newcommand{\CC}{\mathbb{C}}

\newcommand{\LL}{\mathbb{L}}

\newcommand{\XX}{\mathbb{X}}
\newcommand{\ZZ}{\mathbb{Z}}

\newcommand{\Der}[1]{\mathrm{D}^b(#1)}

\newcommand{\Hom}[3]{\mathrm{Hom}_{#1}(#2,#3)}
\newcommand{\End}[2]{\mathrm{End}_{#1}(#2)}
\newcommand{\Ext}[4]{\mathrm{Ext}^{#1}_{#2}(#3,#4)}

\newcommand{\lcm}{\mathrm{l.c.m.}}

\DeclareMathOperator{\CM}{CM}

\newcommand{\coh}[1]{\mathrm{coh}\textrm{-}#1}
\newcommand{\vect}[1]{\mathrm{vect}\textrm{-}#1}
\newcommand{\svect}[1]{\underline{\mathrm{vect}}\textrm{-}#1}
\newcommand{\zvect}[1]{\underline{\mathrm{vect}}^{\mathbb{Z}}\textrm{-}#1}

\newcommand{\mmod}[1]{\mathrm{mod}\textrm{-}#1}

\newcommand{\Cc}{\mathcal{C}}

\newcommand{\Hh}{\mathcal{H}}

\newcommand{\Nn}{\mathcal{N}}
\newcommand{\Oo}{\mathcal{O}}
\newcommand{\Ss}{\mathcal{S}}
\newcommand{\Tt}{\mathcal{T}}

\newcommand{\vx}{\vec{x}}
\newcommand{\vy}{\vec{y}}
\newcommand{\vc}{\vec{c}}

\newcommand{\vw}{\vec{\omega}}
\newcommand{\vom}{\vec{\omega}}

\newcommand{\can}[1]{{(#1)}}

\newcommand{\her}[1]{[#1]}

\newcommand{\fuchs}[1]{{\langle#1]}}

\newcommand{\hstar}[1]{[#1]}
\newcommand{\nak}[2]{N_{#1}(#2)}
\newcommand{\Nak}[2]{\Nn_{#1}(#2)}

\newcommand{\cox}[1]{\mathbf{\Phi}_n}

\newcommand{\rperp}[1]{{#1}^{\bot}}
\newcommand{\incl}{\hookrightarrow}

\newcommand{\X}{\XX}
\newcommand{\Z}{\ZZ}
\newcommand{\cp}{\mathcal{P}}

\newcommand{\bl}{\mathbb{L}}

\newcommand{\co}{{\mathcal O}}

\newcommand{\bS}{{\mathbb{S}}}

\newcommand{\sHom}{\underline{\mathrm{Hom}}}

\begin{document}

\title{Nakayama algebras and Fuchsian singularities}
\author{Helmut Lenzing, Hagen Meltzer and Shiquan Ruan$^*$}

\address{Institut f{\"u}r mathematik, Universit{\"a}t Paderborn, 33095 Paderborn, Germany.} \email{helmut@math.uni-paderborn.de}
\address{Instytut matematyki, Uniwersytet Szczeci{\'n}ski, 70451 Szczecin, Poland.} \email{hagen.meltzer@usz.edu.pl}
\address{School of Mathematical Sciences, Xiamen University, 361005 Fujian, P.R. China.} \email{sqruan@xmu.edu.cn }

\thanks{$^*$ Corresponding author}

\begin{abstract} This present paper is devoted to the study of a class of Nakayama algebras $N_n(r)$ given by the path algebra of the equioriented quiver $\mathbb{A}_n$ subject to the nilpotency degree $r$ for each sequence of $r$ consecutive arrows. We show that the Nakayama algebras $N_n(r)$ for certain pairs $(n,r)$ can be realized as endomorphism algebras of tilting objects in the bounded derived category of coherent sheaves over a weighted projective line, or in its stable category of vector bundles. Moreover, we classify all the Nakayama algebras $N_n(r)$ of Fuchsian type, that is, derived equivalent to the bounded derived categories of extended canonical algebras. We also provide a new way to prove the classification result on Nakayama algebras of piecewise hereditary type, which have been done by Happel--Seidel before.
 \end{abstract}


\date{\today}

\subjclass[2010]{Primary 18G80; 16G20; 16G70; Secondary 14J17; 14H60; 16G60}

\maketitle

\section{Introduction}

%
%
%
%
%

Let $k$ be an algebraically closed field.
A finite dimensional algebra over $k$ is called a \emph{Nakayama algebra} if its indecomposable projective or injective modules have a unique composition series.
A very natural class of such algebras is formed by the algebras $N_n(r)$, given as the path algebra of the equioriented quiver
$$\xymatrix{
  1& \ar[l]_-{x}  2&  \ar[l]_-{x}  3&  \ar[l]_-{x}  \cdots  &\ar[l]_-{x}  n-1& \ar[l]_-{x}  n  }$$
of type $\mathbb{A}_n$ subject to all relations $x^r=0$.
The representation theory of the algebras $N_n(r)$ is well-understood. They are all representation finite and even simply connected. This changes drastically if one asks for a description of their bounded derived categories.


A finite dimensional algebra $\Lambda$ over $k$ is said to be \emph{piecewise hereditary}, if its bounded derived category $D^b(\Lambda)$ is equivalent to $D^b(\mathcal{H})$ of a hereditary abelian category $\mathcal{H}$. The possible types of $\mathcal{H}$ occurring in this
situation have been described in \cite{Happel:1997}. More precisely, either it is the module category of finite
dimensional modules over a finite dimensional basic hereditary $k$-algebra,
or it is the category of coherent sheaves over a weighted projective line in the sense of \cite{Geigle:Lenzing:1987}.
Since $k$ is assumed
to be algebraically closed, for the former case $\Lambda$ is derived equivalent to a path algebra over $k$ of a finite quiver without oriented cycles, while for the latter case, $\Lambda$ is derived equivalent to a canonical algebra in the sense of \cite{Ringel:1984}.
In the first case we call $\Lambda$ of \emph{module type} and
in the second case we call $\Lambda$ of \emph{sheaf type}.
We refer to \cite{Happel:Zacharia:2008, Happel:Zacharia:2010} for properties of piecewise hereditary algebras.

Associated to a weighted projective line $\X$ of negative Euler characteristic, there is a graded algebra $R=\bigoplus_{n\geq0}\Hom{}{\co_{\X}}{ \tau^n\co_{\X}}$, where $\co_{\X}$ stands for the structure sheaf of $\X$ and $\tau$ is the Auslander-Reiten translation in the category of coherent sheaves over $\X$. The graded algebra $R$ is isomorphic to an algebra of automorphic forms, called \emph{Fuchsian singularity}, with respect to certain Fuchsian group action on the upper half plane for the base field of complex numbers, (more precisely, the orbit fundamental group of $\X$ action on its universal covering space), see \cite{Milnor:1975, Neumann:1977}.
The $k$-algebra $R$ is commutative, graded integral
Gorenstein, in particular Cohen–Macaulay, of Krull dimension two, c.f. \cite{Lenzing:1994, Lenzing:2021}.

The graded singularity category of $R$ was considered by Buchweitz \cite{Buchweitz:1987} and Orlov \cite{Orlov:2009}, see also
Krause’s account \cite{Krause:2005} for a related, but slightly different approach.
It is remarkable that this singularity category is equivalent to the stable category $\zvect\X:=\vect\X/[\tau^{\Z}\co_\X]$ of vector bundles over $\X$ modulo the $\tau$-orbit $\tau^{\Z}\co_\X$ of the structure sheaf, and also equivalent to the bounded derived category of the \emph{extended canonical algebra} arising from the canonical algebra (associated with $\X$) via one-point extension by an indecomposable
projective module.
We refer to \cite{Lenzing:2011, Lenzing:Pena:2011} for more details.

In their remarkable paper \cite{Happel:Seidel:2010}, Happel and Seidel classify all the Nakayama algebras which are piecewise hereditary. The aim of this paper is to classify all the Nakayama algebras of \emph{Fuchsian type}, that is, the Nakayama algebras which are derived equivalent to Fuchsian singularities, or equivalently, derived equivalent to extended canonical algebras. The main result of this paper is contained in Figure \ref{Figure1} given below ({see Theorem \ref{classification-of-Fuchsian-nakayama-cat}}).

%


\small
\begin{figure}[H]\label{Figure1}
$$
\tiny
\def\l{\langle}
\def\y{\cellcolor{yellow}}
\def\r{\cellcolor{red}}
\def\gg{\cellcolor{OliveGreen}}
\def\g{\cellcolor{YellowGreen}}
\arraycolsep3.5pt
\begin {array}{c|cccccccccccccc}
r&\g\vdots  &\g\vdots  &\y\vdots  &\r\vdots&\r\vdots&\gg\vdots& & & & & &&   \\
 &\g[2,3,18]&\g[2,3,19]&\y(2,3,19)&\r\l2,3,19]&\r\l2,3,20]&\gg342 & & & & & &&   \\ \noalign{\smallskip}
18 &\g[2,3,17]&\g[2,3,18]&\y(2,3,18)&\r\l2,3,18]&\r\l2,3,19]&\gg36 & & & & & &&   \\ \noalign{\smallskip}
17 &\g[2,3,16]&\g[2,3,17]&\y(2,3,17)&\r\l2,3,17]&\r\l2,3,18]&\gg306 & & & & & &&   \\ \noalign{\smallskip}
16 &\g[2,3,15]&\g[2,3,16]&\y(2,3,16)&\r\l2,3,16]&\r\l2,3,17]&\gg288 & & & & & &&   \\ \noalign{\smallskip}
15&\g[2,3,14]&\g[2,3,15]&\y(2,3,15)&\r\l2,3,15]&\r\l2,3,16]&\gg90 & & & & & &&   \\ \noalign{\smallskip}
14 &\g[2,3,13]&\g[2,3,14]&\y(2,3,14)&\r\l2,3,14]&\r\l2,3,15]&\gg252 & & & & & &&   \\ \noalign{\smallskip}
13 &\g[2,3,12]&\g[2,3,13]&\y(2,3,13)&\r\l2,3,13]&\r\l2,3,14]&\gg234 & & & & & &&   \\ \noalign{\smallskip}
12 &\g[2,3,11]&\g[2,3,12]&\y(2,3,12)&\r\l2,3,12]&\r\l2,3,13]&\gg72 & & & & & &&   \\ \noalign{\smallskip}
11 &\g[2,3,10]&\g[2,3,11]&\y(2,3,11)&\r\l2,3,11]&\r\l2,3,12]&\gg198 & & & & & &&   \\ \noalign{\smallskip}
10&\g[2,3,9]&\g[2,3,10]&\y(2,3,10)&\r\l2,3,10]&\r\l2,3,11]&\gg180 & & & & & &&   \\ \noalign{\smallskip}
9 &\g[2,3,8]&\g[2,3,9]&\y(2,3,9)&\r\l2,3,9]&\r\l2,3,10]&\gg18 & & & & & &&   \\ \noalign{\smallskip}
8 &\g[2,3,7]&\g[2,3,8]&\y(2,3,8)&\r\l2,3,8]&\r\l2,3,9]&\r\l2,3,10]&\r\l2,3,11]&\gg24 & & & &&   \\ \noalign{\smallskip}
7 &\g[2,3,6]&\g[2,3,7]&\y(2,3,7)&\r\l2,3,7]&\r\l2,3,8]&\gg126 & & & & & &&   \\ \noalign{\smallskip}
6 &\g[2,3,5]&\g[2,3,6]&\y(2,3,6)&\y(2,3,7)&\r\l2,3,7]&\r\l2,4,7]&\r\l2,5,7]&\gg12 & & & &&   \\ \noalign{\smallskip}
5&\g[2,3,4]&\g[2,3,5]&\y(2,3,5)&\g[2,3,7]&\r\l2,4,5]&\r\l2,5,5]&\r\l2,5,6]&\r\l2,6,6]&\gg40 & & &&   \\ \noalign{\smallskip}
4 &\g[2,3,3]&\g[2,3,4]&\y(2,3,4)&\y(2,4,4)&\y(2,4,5)&\r\l2,4,5]&\r\l2,5,5]&\r\l2,5,6]&\r\l2,5,7]&\gg12  & &&  \\ \noalign{\smallskip}
3&\g[2,3,2]&\g[2,3,3]&\g[2,3,4]&\g[2,3,5]&\y(2,3,5)&\y(2,3,6)&\y(2,3,7)&\r\l2,3,7]&\r\l2,3,8]&\r\l2,3,9]&\r\l2,3,10]& &\gg18\\ \noalign{\smallskip}\hline
 &2& 3&4 &5& 6& 7&8 &9 &10& & && n-r \end {array}
$$
\label{figure:green:wall}
\caption{Module type (green), sheaf type (yellow), Fuchsian type (red) and the dark green wall}
\end{figure}
\normalsize

%

For a triple $a,b,c$, the symbols $[a,b,c],(a,b,c)$ and $\langle a,b,c]$ in Figure \ref{Figure1} each denotes the derived equivalent type of $N_n(r)$ for $r\geq 3$ and $n-r\geq 2$:
\begin{itemize}
  \item [-] module type $[a,b,c]$: denotes the type of the bounded derived category of the hereditary star quiver with three branches of length $a,b,c$ respectively.
  \item [-] sheaf type $(a,b,c)$: denotes the type of the bounded derived category of coherent sheaves $D^b(\coh\X)$ for the weighted projective line $\X:=\X(a,b,c)$.
  \item [-] Fuchsian type $\langle a,b,c]$: denotes the type of the stable category of vector bundles $\zvect\X$ for the weighted projective line $\X:=\X(a,b,c)$.
\end{itemize}
The dark green bricks on the right, forming the wall, contain the Coxeter numbers of $N_n(r)$.

From Figure \ref{Figure1} we see that the Nakayama algebras of Fuchsian type (red area) contain two (infinite) families: $\nak{r+5}{r} \;(r\geq 7)$ and $\nak{r+6}{r}\;(r\geq 6)$,  having Fuchsian type $\langle 2,3,r]$ and $\langle 2,3,r+1]$ respectively; and some further (finite) sporadic cases.

In order to obtain these two infinite families of Fuchsian Nakayama categories(=bounded derived categories of Nakayama algebras of Fuchsian type), we first construct a tilting complex in the bounded derived category of coherent sheaves over the weighted projective line of weight type $(2,3,r)$, whose endomorphism algebra is isomorphic to the Nakayama algebra $N_{r+4}(r)$ (see Proposition \ref{tilting realization}); and then by using one-point extension approach to show that $\nak{r+5}{r}$ have Fuchsian type $\langle 2,3,r]$ (see Corollary \ref{familis Nakayama corollary}), and $\nak{r+6}{r}$ have Fuchsian type $\langle 2,3,r+1]$ (see Proposition \ref{N-r+6-r}). During this procedure, an important strategy is to endow the projective, injective and simple $N_{r+4}(r)$-modules with the degree-rank data, inherited from that of coherent sheaves (see Remark \ref{identity}).


The most difficult part of this paper is to deal with the sporadic Nakayama algebras of Fuchsian types appearing in Figure \ref{Figure1}, namely, to show those sporadic Nakayama algebras have the desired Fuchsian types. For this we use two further strategies. The first one is to develop the extended Happel-Seidel symmetries between distinct Nakayama algebras (see Proposition \ref{HS-symmetry}); and the second one is to  construct tilting objects consisting of line bundles in the stable category $\zvect\X$ for certain weight types, whose endomorphism algebras are Nakayama algebras (see Theorem \ref{tiltingobjects}).


Recall that there is \emph{Happel-Seidel symmetry} developed in \cite{Happel:Seidel:2010}. More precisely, for any $a,b\geq 2$, there is an equivalence $$D^b(N_n(a))\simeq D^b(N_n(b))$$ for $n=(a-1)(b-1)$. This symmetry has a nice explanation from the knowledge of weighted projective lines, since both of these two categories are equivalent to the stable category of vector bundles
$\svect\X:=\vect\X/[\mathcal{L}]$, obtained from $\vect\X$ by factoring out all line bundles $\mathcal{L}$,
for the weighted projective line $\X$ of weight type $(2,a,b)$, c.f. \cite{Kussin:Lenzing:Meltzer:2013adv}. By analysing the realization of Nakayama algebras in the stable category $\svect\X$, we obtain further symmetries between Nakayama algebras, called \emph{extended Happel-Seidel symmetry} as follows (see Proposition \ref{HS-symmetry}):
$$D^b(N_{n\pm1}(a))\simeq  D^b(N_{n\pm1}(b)).$$

With the help of the (extended) Happel-Seidel symmetries, we provide a new (simple) proof for the classification result of piecewise hereditary Nakayama categories obtained in \cite{Happel:Seidel:2010}, by taking perpendicular category approach (see Theorem \ref{classification-of-piecewise-hereditary-nakayama-cat}).

We will give a brief outline of this paper.
In Section 2 we recall some basic materials and explain some notations. In Section 3, we construct a tilting complex in the bounded derived category of coherent sheaves over the weighted projective line of weight type $(2,3,r)$, whose endomorphism algebra is isomorphic to the Nakayama algebra $N_{r+4}(r)$. Then by using perpendicular category and one-point extension approaches we obtain four families of Nakayama algebras of module type or Fuchsian type.
In Section 4, we state the extended Happel-Seidel symmetries between distinct Nakayama algebras, and then
provide a new proof for the classification result of piecewise hereditary Nakayama categories obtained in \cite{Happel:Seidel:2010}. In order to obtain the sporadic Fuchsian Nakayama categories in Figure \ref{Figure1}, we investigate the stable category $\zvect\X$ of vector bundles in more details in Section 5. We finally construct tilting objects consisting of line bundles in $\zvect\X$ for weight types $(2,4,5),(2,4,7),(2,5,5)$ and (2,5,6) respectively, whose endomorphism algebras are Nakayama algebras. The proofs of the extension-free properties need to be considered case by case, which are put in Appendix \ref{Appendix A}. In the final Section 6, we give the explicit formulas of the Coxeter polynomials for a class of Nakayama algebras, and then
obtain the main classification result for Nakayama algebras of Fuchsian type, under certain mild conjectures.

We denote by $\mmod \Lambda$ the category of finitely generated left $\Lambda$-modules, and denote by $D^b(\Lambda)\simeq  D^b(\Lambda')$ the derived equivalence between $\Lambda$ and $\Lambda'$. The Serre functor of a triangulated category is denoted by $\bS=\tau[1]$, where $\tau$ is the Auslander-Reiten translation and $[1]$ denotes the suspension functor. For unexplained representation-theoretic and derived category terminology, we
refer to \cite{AuslanderReitenSmalo1995, Happel:1988} and \cite{Ringel:1984}.

\bigskip
\section{Preliminary}

\subsection{The category of coherent sheaves $\coh\X$}

Let $k$ be an algebraically closed field and $\X=\X(p_1, p_2, p_3)$ be a weighted projective line over $k$ of weight type $(p_1, p_2, p_3)$, where $p_i\geq 2$ is an integer for each $i$.
Following \cite{Geigle:Lenzing:1987}, the category $\coh\X$ of coherent sheaves on $\X$ is obtained by applying Serre's construction to the graded triangle singularity $x_1^{p_1}+x_2^{p_2}+x_3^{p_3}$.

More precisely, the Picard group of $\X$ is naturally isomorphic to the rank one abelian group $\bl=\bl(p_1,p_2,p_3)$ on generators $\vx_1,\vx_2,\vx_3$ subject to the relations $p_1\vx_1=p_2\vx_2=p_3\vx_3:=\vc$, where $\vc$ is called the \emph{canonical element}. The coordinate algebra
\begin{equation}\label{Coordinate algebra S}S:=k[x_1,x_2,x_3]/(x_1^{p_1}+x_2^{p_2}+x_3^{p_3})\end{equation}
of $\X$ is then $\LL$-graded with $\deg(x_i)=\vx_i$ for each $i$. A quick way to arrive at the category $\coh\X$ of coherent sheaves on $\X$ is by forming the Serre quotient $$\coh\X:=\frac{\text{mod}^{\LL}S}{\text{mod}^{\LL}_0S},$$ where $\text{mod}^{\LL}S$ (\emph{resp.} $\text{mod}^{\LL}_0S$) is the category of finitely generated $\LL$-graded modules over $S$ (\emph{resp.} those of finite length).

The category $\coh\X$ satisfies the Serre duality of the form $$D\Ext{1}{}{X}{Y}\cong\Hom{}{Y}{X(\vec{\omega})}$$ functorially in $X,Y\in\coh\X$,
where $\vw=\vc-\sum_{i=1}^3\vx_i$ is called the \emph{dualizing element} in $\LL$. This implies
the existence of almost split sequences in $\coh\X$ with the
Auslander--Reiten translation $\tau$ given by the degree shift with
$\vec{\omega}$.
By $\vect\X$ we denote the full subcategory of the category $\coh\X$ formed by all vector bundles, that is, all locally free sheaves.

Denote by $p={\rm l.c.m.}(p_1,p_2,p_3)$. There is a group homomorphism $\delta\colon \bl(p_1,p_2,p_3)\rightarrow \mathbb{Z}$ given by $\delta(\vec{x}_i)=\frac{p}{p_i}$ for $1\leq i\leq 3$.
The complexity of $\coh\X$ (or $\vect\X$) with regard to the classification of indecomposable objects is largely determined by the \emph{(orbifold) Euler characteristic} of $\X$, given by the expression $$\chi_{\X}=-\frac{1}{p}\delta(\vw)=2-\sum_{i=1}^3(1-\frac{1}{p_i}).$$


Each element $\vx$ in $\LL$ has the normal form $\vx=\sum_{i=1}^3l_i\vx_i+l\vc$, where $0\leq l_i\leq p_i-1$ for each $i$ and $l\in\ZZ$. Up to isomorphisms, the line bundles in $\coh\X$ are given by the system $\mathcal{L}$ of twisted structure sheaves $\co(\vx)$ with $\vx\in\bl$.
The Auslander--Reiten translation $\tau$ acts on $\mathcal{L}$, and the index $[\bl:\Z\vw]$ counts the number of $\tau$-orbits of line bundles.

There is a canonical tilting sheaf $T_{\text{can}}=\bigoplus_{0\leq\vx\leq \vc}\co(\vx)$ in $\coh\X$, whose endomorphism algebra is isomorphic to the canonical algebra $C(p_1,p_2,p_3)$ in the sense of Ringel \cite{Ringel:1984}. Consequently, the Grothendieck group $K_0(\X)$ of $\coh\X$ is the free abelian group generated by the isomorphism classes $[\co(\vx)]$ for $0\leq \vx\leq \vc$. The \emph{determinant} function on $K_0(\X)$ (or on $\coh\X$) is defined via $\det(\co(\vx))=\vx$, while the \emph{degree} function is determined by $\deg(\co(\vx))=\delta(\vx)$ and the \emph{rank} function is determined by $\text{rank}(\co(\vx))=1$, for any $\vx\in\mathbb{L}$.

\subsection{$\mathcal{L}$-exact structure on $\vect\X$}

 By \cite{Kussin:Lenzing:Meltzer:2013adv}, the category $\CM^{\mathbb{L}}S$ of maximal Cohen--Macaulay $\mathbb{L}$-graded modules over $S$ is equivalent to the category $\vect\X$ of vector bundles, which induces a distinguished exact structure on $\vect\X$. More precisely, an exact sequence
 $\xymatrix{0\ar[r]& E_1\ar[r]& E_2\ar[r]& E_3\ar[r]& 0}$ in $\vect\X$ is called \emph{distinguished exact} if $$\xymatrix{0\ar[r]&\Hom{}{L}{E_1}\ar[r]&\Hom{}{L}{E_2}\ar[r]&\Hom{}{L}{E_3}\ar[r]& 0}$$ is an exact sequence for any line bundle $L$.

Recall that an exact category in the sense of Quillen is called a \emph{Frobenius category} if it has enough projective objects and injective objects, and the projectives and the injectives coincide.
Under the distinguished exact structure, $\vect\X$ becomes a Frobenius category with indecomposable projective-injectives the system of all line bundles $\mathcal{L}$.
Therefore, the corresponding stable category $\svect\X:=\vect\X/[\mathcal{L}]$ is a triangulated category, see \cite{Happel:1988}. Here, $[\mathcal{L}]$ denotes the two-sided ideal of morphisms in $\vect\X$ factoring through $\text{add}(\mathcal{L})$.

\subsection{$\tau^{\ZZ}\co$-exact structure on $\vect\X$}

For $\chi_{\X}\neq 0$, by restricting the grading of $S$ to the subgroup $\ZZ\vw$, we obtain a graded surface singularity $R=\bigoplus_{n\in\ZZ}R_{n}$, where $R_{n}=S_{n\vw}$, see \cite{Geigle:Lenzing:1991} for details. This algebra $R$ turns out to be isomorphic--as a $\ZZ$-graded algebra--either to the algebra of invariants of the action of a binary polyhedral group on the polynomial algebra $\CC[X,Y]$ (for $\chi_{\X}> 0$), or else to the algebra of entire automorphic forms with regard to an action of a Fuchsian group of the first kind on the upper complex half plane (for $\chi_{\X}< 0$), see for instance \cite{Lenzing:1994}.

The category $\CM^{\ZZ}R$ of maximal Cohen--Macaulay $\ZZ$-graded modules over $R$ is equivalent to the category $\vect\X$ of vector bundles, which induces a $\tau^{\ZZ}\co$-exact structure on $\vect\X$. Here, $\co:=\co_{\X}$ denotes the structure sheaf over $\X$, $\tau$ is the Auslander-Reiten translation of $\coh\X$, and $\tau^{\ZZ}\co$ denotes the $\tau$-orbit of $\co$. More precisely, an exact sequence $\xymatrix{0\ar[r]& E_1\ar[r]& E_2\ar[r]& E_3\ar[r]& 0}$ in $\vect\X$ is called \emph{$\tau^{\ZZ}\co$-exact} if $$\xymatrix{0\ar[r]&\Hom{}{\tau^{n}\co}{E_1}\ar[r]&\Hom{}{\tau^{n}\co}{E_2}\ar[r]&
\Hom{}{\tau^{n}\co}{E_3}\ar[r]& 0}$$ is an exact sequence for any $n\in\ZZ$.

Under $\tau^{\Z}\co$-exact structure, $\vect\X$ also becomes a Frobenius category, whose indecomposable projective-injectives are the $\tau$-orbit $\tau^{\Z}\co$ of the structure sheaf.
Therefore, the corresponding stable category $\zvect\X:=\vect\X/[\tau^{\Z}\co]$ is also a triangulated category.

\subsection{Nakayama algebra and Nakayama category}\label{Nakayama categories}

A finite dimensional algebra over $k$ is called a \emph{Nakayama algebra} if its indecomposable projective or injective modules are uniserial, that is, have a unique composition series.
A very natural class of such algebras is formed by the algebras $N_n(r)$, given as the path algebra of the equioriented quiver
$$\xymatrix{
  1& \ar[l]_-{x}  2&  \ar[l]_-{x}  3&  \ar[l]_-{x}  \cdots  &\ar[l]_-{x}  n-1& \ar[l]_-{x}  n  }$$
of type $\mathbb{A}_n$ subject to all relations $x^r=0$.
The representation theory of the algebras $N_n(r)$ is well-understood. This changes drastically if one asks for a description of the bounded derived category $\Nak{n}{r}:=\Der{N_n(r)}$, which is called a \emph{Nakayama category} in this paper. The following four types of Nakayama categories will be considered later.

For a triple $a,b,c$, the symbols $[a,b,c],(a,b,c),\langle a,b,c]$ and $\langle a,b,c\rangle$ each denotes the derived equivalence type of some triangulated category, called of \emph{module type}, of \emph{sheaf type}, of \emph{Fuchsian type} or of \emph{triangle type}, respectively.

\begin{itemize}
  \item [-] $[a,b,c]$: denotes the type of the bounded derived category of the hereditary star quiver with three branches of length $a,b,c$ respectively.
  \item [-] $(a,b,c)$: denotes the type of the bounded derived category of coherent sheaves $D^b(\coh\X)$ for the weighted projective line $\X:=\X(a,b,c)$.
  \item [-] $\langle a,b,c]$: denotes the type of the stable category of vector bundles $\zvect\X$ for the weighted projective line $\X:=\X(a,b,c)$.
  \item [-] $\langle a,b,c\rangle$: denotes the type of the stable category of vector bundles $\svect\X$
      for the weighted projective line $\X:=\X(a,b,c)$.
\end{itemize}

We say that a Hom-finite triangulated category $\Tt$ is \emph{derived hereditary} if $\Tt$ is triangle equivalent to the category $\Der{\mmod{H}}$, where $H$ is a finite dimensional hereditary $k$-algebra. More generally, we say that $\Tt$ is \emph{piecewise hereditary} if there exists a Hom-finite hereditary abelian $k$-category $\Hh$ such that $\Tt$ is triangle equivalent to the category $\Der{\Hh}$. By a result of Happel~\cite{Happel:2001}, a piecewise hereditary $k$-category $\Tt$ with a tilting object is derived hereditary or is triangle equivalent to the bounded derived category $\Der{\coh\XX}$ of coherent sheaves on a weighted projective line. Piecewise hereditary categories with a tilting object, equivalently piecewise hereditary $k$-algebras, have been investigated, for instance, in \cite{Happel:Seidel:2010,Happel:Zacharia:2008,Happel:Zacharia:2010}.

\subsection{One-point extension of an algebra}

We recall the concept of a one-point extension of an algebra $\Lambda$. Let $M\in\mmod\Lambda$. Then the one-point extension of $\Lambda$ by $M$ is
denoted by $\Lambda [M]$ and defined as the triangular matrix ring
$$\Lambda [M]=\left[\begin{array}{cc}\Lambda&M\\
0&k\end{array}\right],$$
with multiplication given by
$$\left( \begin{array}{cc}\lambda&\alpha\\ 0&c \end{array}\right)
\left( \begin{array}{cc}\lambda' &\alpha' \\ 0&c' \end{array} \right)
=\left( \begin{array}{cc}\lambda\lambda'&\lambda \alpha'+\alpha c\\ 0&cc' \end{array} \right)$$
for $\lambda,\lambda'\in\Lambda, \alpha,\alpha'\in M$ and $c,c'\in k$. For details we refer to \cite{Ringel:1984}.

We emphasize that a derived equivalence between two algebras can be extended to one-point extensions, see \cite{Barot-Lenzing:2003}. More precisely, assume there is an equivalence $\xymatrix{D^b(\Lambda)\ar[r]^-{\simeq}&  D^b(\Lambda'),}$ which sends $M\in\mmod\Lambda$ to $M'\in\mmod\Lambda'$, then $D^b(\Lambda[M])\simeq  D^b(\Lambda'[M'])$.
As applications to Nakayama algebras, we obtain that $N_{n+1}(r)=N_{n}(r)[I_{n+2-r}]$ is derived equivalent to $N_{n}(r)[P_{n+2-r}]$, where $I_{n+2-r}$ (resp. $P_{n+2-r}$) is the projective (resp. injective) $N_{n}(r)$-module corresponding to the vertex $n+2-r$, see for example \cite{Happel:Seidel:2010}.

The one-point extension of canonical algebras is an important subject in this paper. Let $P$ be an indecomposable projective module over a canonical algebra $C=C(p_1,p_2,p_3)$.
The one-point extension $A = C[P]$ is
called an \emph{extended canonical algebra}. It does not matter which indecomposable projective one takes for the one-point extension. These algebras all happen to be derived equivalent. Moreover, their bounded derived categories are equivalent to the stable category $\zvect\X$ for the weighted projective line $\X=\X(p_1,p_2,p_3)$ when $\frac{1}{p_1}+\frac{1}{p_2}+\frac{1}{p_3}< 1$.

\subsection{Perpendicular category}

Recall from \cite{Simson-Skowronski:2007} that for any $\Lambda$ module $M$, the full subcategory
$$M^{\perp}:=\{X\in\mmod\Lambda\,|\,\Hom{\Lambda}{M}{X}=0=\Ext{1}{\Lambda}{M}{X}\}$$
of $\mmod\Lambda$ is called the \emph{right perpendicular category} of $M$. For example, the right perpendicular category $P_n^{\perp}$ of $\mmod N_{n}(r)$ with respect to the last indecomposable projective module $P_n$ is equivalent to $\mmod N_{n-1}(r)$. We refer to \cite{Geigle:Lenzing:1991} for the perpendicular category of a hereditary abelian category.

For a triangulated Hom-finite $k$-category $\Tt$, we have the right perpendicular category $\rperp{\{E_1,\ldots,E_r\}}$ with regard to an exceptional sequence $(E_1,\ldots,E_r)$ in $\Tt$, which is a triangulated subcategory of $\Tt$. The full embedding $\rperp{\{E_1,\ldots,E_r\}}\incl \Tt$ then has a left adjoint, see~\cite{Bondal:Kapranov:1989}.

Recall that there is a canonical tilting sheaf in $\coh\X$ consisting of line bundles for $\X=\X(p_1,p_2,p_3)$, whose endomorphism algebra is isomorphic to the canonical algebra $C=C(p_1,p_2,p_3)$. Therefore, we have an equivalence $D^b(\coh\X)\simeq  D^b(\mmod C)$. Since the group $\LL$ acts on line bundles transitively, the right perpendicular category of $D^b(\coh\X)$ with respect to any line bundle is equivalent to bounded derived category of the hereditary star quiver with three branches of length $p_1,p_2,p_3$ respectively, hence of type $[p_1,p_2,p_3]$.



\subsection{Coxeter polynomial}\label{Coxeter polynomial section}

Let $\{P(i)|1\leq i\leq n\}$ be a complete set of representatives from the isomorphism classes of the indecomposable projective $\Lambda$-modules.
The \emph{Cartan matrix} $C=C_{\Lambda}$ is defined to be the integer-valued $n\times n$-matrix with entries $$c_{ij}=\textup{dim}_k\textup{Hom}_{\Lambda}(P(i),P(j)).$$
The \emph{Coxeter matrix} $\Phi_{\Lambda}$ is defined to be the matrix $$\Phi_{\Lambda}=-C^{-t}C.$$
The characteristic polynomial $\chi_{\Lambda}$ of $\Phi_{\Lambda}$: $$\chi_{\Lambda}=\textup{det}(\lambda E_{n}-\Phi_{\Lambda})$$ is called the \emph{Coxeter polynomial} of $\Lambda$, where $E_{n}$ is the identity matrix of size $n\times n$.
The minimal number $m$ such that $\Phi_{\Lambda}^m=E_{n}$ is called the \emph{Coxeter number} of $\Lambda$.

The Coxeter polynomial plays an important role for understanding the structure of the module category and its bounded derived category for an algebra. We refer to \cite{Ringel:1994, Happel:1997, Lenzing:1999, Lenzing:Pena:2008} for properties of Coxeter polynomials.

For the Nakayama algebra $N_n(r)$, we denote its Cartan matrix by $\Phi_{(n,r)}$ and its Coxeter polynomial by $\chi_{(n,r)}$. Then by direct computations, we have
\begin{itemize}
  \item[-] $\chi_{(17,8)}=(\lambda+1)(\lambda^{16}+\lambda^{8}+1)$;
  \item[-] $\chi_{(16,3)}=(\lambda+1)(\lambda^{6}-\lambda^{3}+1)(\lambda^{9}+1)$;
  \item[-] $\chi_{(15,6)}=\chi_{(15,4)}=(\lambda+1)(\lambda^{8}+\lambda^{4}+1)(\lambda^{6}+1)$;
  \item[-] $\chi_{(15,5)}=(\lambda+1)(\lambda^4+1)(\lambda^5+1)^2$;
  \item[-] $\chi_{(14,7)}=(\lambda+1)(\lambda^{6}-\lambda^{3}+1)(\lambda^{7}+1)$.
  \end{itemize}
Consequently, the Coxeter numbers for $N_{(17,8)}, N_{(16,3)}, N_{(15,6)}, N_{(15,5)}, N_{(15,4)}, N_{(14,7)}$
are given by $24, 18, 12, 40, 12, 126$ respectively.

\section{Nakayama categories associated to a tilting complex}

In this section, we first construct a tilting complex in the category $\coh\X$ of coherent sheaves of weight type $(2,3,r)$ with $r\geq 4$, whose endomorphism algebra is isomorphic to the Nakayama algebra $\nak{r+4}{r}$. Then by the technique of one-point extension, we show that the Nakayama algebras $\nak{r+5}{r}\, (r\geq 7)$ and $\nak{r+6}{r}\, (r\geq 6)$ are of Fuchsian type; while by considering the perpendicular subcategories, we obtain that $\nak{r+3}{r}$ and $\nak{r+2}{r}$ for $r\geq 3$ are of module type.



The following result is essential in this section. Recall that the Serre functor of $\Der{\coh\XX}$ is denoted by $\bS=\tau[1]$.


\begin{proposition}\label{tilting realization}
Let $\XX$ be a weighted projective line of weight type $(2,3,r)$ with $r\geq4$.
Then the system
 \begin{align}
\begin{split}\label{useful tilting complex}
\bS^{-1}(&\Oo((r-4)\vx_3))\to \bS^{-1}(\Oo((r-3)\vx_3))\to \bS^{-1}(S_{20})\to\\ &\Oo\to\Oo(\vx_3)\to
\cdots\to\Oo((r-4)\vx_3)\to\Oo((r-3)\vx_3)\to S_{20}\to \bS\Oo\to
\bS(\Oo(\vx_3))
 \end{split}
 \end{align}
 forms a tilting complex in $\Der{\coh\XX}$ whose endomorphism algebra is isomorphic to the Nakayama algebra $\nak{r+4}{r}$.
Consequently, $\nak{r+4}{r}$ is derived equivalent to $\coh\XX$.
\end{proposition}

\begin{proof}
Set $$Y_0=\bigoplus\limits_{k=0}^{r-3}\Oo(k\vx_3),\quad Y_1=\bS(\Oo\oplus
\Oo(\vx_3)),\quad Y_{-1}=\bS^{-1}(\Oo((r-4)\vx_3)\oplus \Oo((r-3)\vx_3)).$$ We only need to show that $$T=Y_{-1}\oplus \bS^{-1}(S_{20})\oplus Y_0\oplus S_{20}\oplus Y_1$$ forms a tilting complex in $\Der{\coh\XX}$.

For any line bundle $L$ and any element $\vy\in\bl$,
it is well-known that $L\oplus L(\vy)$ is extension-free if and only if $-\vc\leq \vy\leq \vc$, hence each $Y_i$ is extension-free for $-1\leq i\leq 1$.
For $0\leq a<b\leq r-3$, by Serre duality we have $$\Hom{D^b(\coh\X)}{\co(a\vx_3)}{\bS(\co(b\vx_3))}=D\Hom{D^b(\coh\X)}{\co(b\vx_3)}{\co(a\vx_3)}=0.$$ Moreover, since $\coh\X$ is a hereditary category, we have $\Hom{D^b(\coh\X)}{X}{\bS^kY}=0$ for any $k\geq 2$ and $X,Y\in\coh\X$. Then by using Lemma \ref{extension-free formula for tau^k U[k]} and Lemma \ref{extension-free formula for tau^k U[k]: general case}, we obtain that $T$ is extension-free.

Note that the number of indecomposable direct summands of $T$ coincides with the rank of the Grothendieck group of $\Der{\coh\XX}$, hence $T$ is a tilting complex in $\Der{\coh\XX}$. It is straightforward to show that the endomorphism algebra of $T$ is isomorphic to $\nak{r+4}{r}$. Then the proof is finished.
\end{proof}

\begin{remark}\label{identity}
Let $\Psi$ be the derived equivalence induced by the tilting complex $T$ in \eqref{useful tilting complex}:
\begin{equation}\label{phi}\xymatrix{\Psi:\Der{\coh\XX(2,3,r)}\ar[r]^-{\simeq}& \Der{\nak{r+4}{r}}.}
\end{equation}
Then each indecomposable projective $\nak{r+4}{r}$-module $P_i$ corresponds to the direct summand $T_i$ of $T$, and each indecomposable injective module $I_i=\bS(P_i)$ corresponds to $\bS(T_i)$. By using projective resolution, we obtain that the simple $\nak{r+4}{r}$-module $S_{i}$ (for $5\leq i\leq r$) corresponds to the torsion sheaf $S_{3,i-4}$ in $\coh\XX(2,3,r)$.

Moreover, the functions defined on $\Der{\coh\XX(2,3,r)}$ or $K_0(\coh\X)$, such as rank and degree functions, can take values for $\nak{r+4}{r}$-modules in a natural way. For example, the rank of projective or injective $\nak{r+4}{r}$-modules are given by
$${\rm{rank}}(P_i)=-{\rm{rank}}(I_i)=
         \left\{\begin{array}{lll}
         -1, && i=1,2,r+3,r+4;\\
                     0, && i=3,r+2;\\
             1, && 4\leq i\leq r+1;
             \end{array} \right.$$
while the rank data for the simple modules are $${\rm{rank}}(S_i)=
         \left\{\begin{array}{lll}
         -1, && i=1,r+1,r+2;\\
          1, && i=3,4,r+4;\\
          0, && {\text{otherwise}}.
             \end{array} \right.$$
\end{remark}


As a consequence, we can obtain the following result due to Happel and Seidel.

\begin{corollary}[\cite{Happel:Seidel:2010}]\label{familis Nakayama corollary}
The following statements hold.
\begin{itemize}
\item[(1)] For $r\geq 3$, $\nak{r+3}{r}$ is derived hereditary of type $\hstar{2,3,r}$.
\item[(2)] For $r\geq 3$, $\nak{r+2}{r}$ is derived hereditary of type $\hstar{2,3,r-1}$.
\item[(3)] For $r\geq 7$, $\nak{r+5}{r}$ has Fuchsian type $\fuchs{2,3,r}$.
\end{itemize}
\end{corollary}

\begin{proof}
For the first two statements, we recall that there exists a tilting complex $\widetilde{T}_{\rm{can}}=\bigoplus_{0\leq \vx\leq \vc}\bS(\Oo(\vx-\vc+\vx_3))$ obtained from the canonical tilting sheaf $T_{\rm{can}}=\bigoplus_{0\leq \vx\leq \vc}\Oo(\vx)$ by degree shift and suspension shift.
Note that $\widetilde{T}_{\rm{can}}$ and the tilting complex in Proposition \ref{tilting realization} share the same last two indecomposable direct summands. Therefore, by taking the right perpendicular category for the last member $\bS(\Oo(\vx_3))$, we obtain that $\nak{r+3}{r}$ is derived hereditary of type $\hstar{2,3,r}$; while considering the right perpendicular category for the last two members $\bS\Oo\oplus \bS(\Oo(\vx_3))$ we obtain that
$\nak{r+2}{r}$ is derived hereditary of type $\hstar{2,3,r-1}$.

For (3), observe that $\nak{r+5}{r}=\nak{r+4}{r}[I_6]$ is derived equivalent to $\nak{r+4}{r}[P_6]$ by \cite{Happel:Seidel:2010}, where $I_6$ (\emph{resp}. $P_6$) is the injective (\emph{resp}. projective) $\nak{r+4}{r}$-module attached to the vertex 6. 
By Proposition \ref{tilting realization}, $P_6$ corresponds to the line bundle $\co(2\vx_3)$ in $\coh\XX(2,3,r)$ under the equivalence $\Psi$ in Remark \ref{identity}. Then by \cite{Lenzing:Pena:2011}, $\nak{r+5}{r}$ has Fuchsian type $\fuchs{2,3,r}$.
 \end{proof}

\begin{remark}\label{why r geq 7}
During the proof of the statement (3) in the above corollary, we obtain that the projective $\nak{r+4}{r}$-module $P_6$ corresponds to a line bundle (i.e., $\co(2\vx_3)$), and then $\nak{r+5}{r}$ is equivalent to the extended canonical algebra of $C(2,3,r)$, which only requires the condition $r\geq 5$.

But only when $r\geq 7$, the extended canonical algebra of $C(2,3,r)$ has Fuchsian type $\fuchs{2,3,r}$. In fact,
for $r=5$, $\nak{10}{5}$ is derived equivalent to the extended canonical algebra of $C(2,3,5)$, which has module type $[2,3,7]$ by \cite[Proposition 3.3]{Lenzing:Pena:2011}; while for $r=6$, $\nak{11}{6}$ is derived equivalent to the extended canonical algebra of $C(2,3,6)$, which has sheaf type $(2,3,7)$ by \cite[Proposition 3.5]{Lenzing:Pena:2011}.
\end{remark}

By using the technique of one-point extension for $\nak{r+4}{r}$ twice, we find a new family of Nakayama algebras of Fuchsian type.

\begin{proposition}\label{N-r+6-r}
For $r\geq 6$, $\nak{r+6}{r}$ has Fuchsian type $\fuchs{2,3,r+1}$.
\end{proposition}

\begin{proof}
During the proof, we denote by $A=N_{r+6}(r)$ for convenience.

Obviously, the injective module $\bigoplus_{i=1}^{r+6}I_i$ is a tilting module in $\mmod A$. Note that $I_{r+6}$ and $\tau I_{r+6}$ are both simple $A$-modules.
Hence, we obtain an APR tilting $A$-module $$T=(\bigoplus_{i=1}^{r+5}I_i)\oplus \tau I_{r+6}.$$
Let $\Gamma=\End{A}{T}$. Then we have $D^b(A)\simeq D^b(\Gamma)$.
Set $$T'=T\backslash{I_{r+5}}=(\bigoplus_{i=1}^{r+4}I_i)\oplus \tau I_{r+6} \quad\text{and}\quad T''=T'\backslash {\tau I_{r+6}}=\bigoplus_{i=1}^{r+4}I_i.$$ Denote by
$\Gamma'=\End{A}{T'}$ and $\Gamma''=\End{A}{T''}$, respectively. Then we have
$\Gamma=\Gamma'[M]$ and $\Gamma'=\Gamma''[N]$, where $M=\Hom{A}{T'}{I_{r+5}}$ and $N=\Hom{A}{T''}{\tau I_{r+6}}$.

Observe that $\Gamma''\cong N_{r+4}(r)$, which yields an equivalence
$$\xymatrix{\Phi'':D^b(\Gamma'') \ar[r]^-{\simeq}&  \Der{\coh\XX(2,3,r)},}$$ sending projective $\Gamma''$-modules $\Hom{A}{T''}{I_i}$ (for $1\leq i\leq r+4$) to the indecomposable direct summands of the tilting complex \eqref{useful tilting complex}.

For $1\leq i\leq r+4$, $\Hom{A}{I_i}{\tau I_{r+6}}\neq 0$ if and only if $i=6$. Hence $N$ is a simple $\Gamma''$-module (associated to the vertex $6$), which corresponds to the torsion sheaf $S_{3,2}$ in $\coh\XX(2,3,r)$ under $\Phi''$ by Remark \ref{identity}. Consequently, we obtain an equivalence $$\xymatrix{\Phi':D^b(\Gamma')=D^b(\Gamma''[N])\ar[r]^-{\simeq}& \Der{\coh\XX(2,3,r+1)}.}$$ Moreover, we have the following commutative diagram:
$$\xymatrix{D^b(\Gamma'')\ar[r]^-{\iota_1}\ar[d]_-{\Phi''}^-{\simeq}& D^b(\Gamma')\ar[d]_-{\simeq}^-{\Phi'}\\
 \Der{\coh\XX(2,3,r)}\ar[r]^-{\iota_2}& \Der{\coh\XX(2,3,r+1)},}$$
where $\iota_1$ (\emph{resp.} $\iota_2$) is the natural embedding induced by the one-point extension w.r.t $N$ (\emph{resp.} $S_{3,2}$).

Let $M'=\Hom{A}{T''}{I_{r+5}}\in\mmod \Gamma''$. Clearly, we have $\iota_1(M')=M$.
Note that
\begin{equation}\label{Hom Ii to I r+5}
 \Hom{A}{I_i}{I_{r+5}}=
  \left\{\begin{array}{lll}
    0, && 1\leq i\leq 5;\\
    k, && 6\leq i\leq r+4.
  \end{array} \right.
\end{equation}
Hence, the Jordan-H\"older factors of the $\Gamma''$-module $M'$ are the simple modules $S_{i}$ for $6\leq i\leq r+4$. It follows from Remark \ref{identity} that  $${\rm{rank}}(M')=\sum\limits_{i=6}^{r+4}{\rm{rank}}(S_{i})=-1.$$
Combining with \eqref{Hom Ii to I r+5}, we know that $\Phi''(M')=L'[1]$ for some line bundle  $L'\in\coh\XX(2,3,r)$.

By \cite{Geigle:Lenzing:1991} we know that $\iota_2$ is rank-preserving. Therefore, $L:=\iota_2(L')$ is a line bundle in $\coh\XX(2,3,r+1)$, and $$\Phi'(M)=\Phi'\circ\iota_1(M')=\iota_2\circ\Phi''(M')=L[1]\in D^b(\coh\XX(2,3,r+1)).$$
Let $\widetilde{T}_{\rm{can}}=\bigoplus_{0\leq \vx\leq \vc}L(\vx)[1]$. Then $\widetilde{T}_{\rm{can}}$ is a tilting complex in $D^b(\coh\XX(2,3,r+1))$, whose endomorphism algebra is isomorphic to the canonical algebra $C(2,3,r+1)$. Moreover, $\widetilde{T}_{\rm{can}}$ induces an equivalence:
$$\xymatrix{D^b(\Gamma') \ar[r]^-{\Phi'}_-{\simeq}& \Der{\coh\XX(2,3,r+1)}\ar[rr]_-{\simeq}^-{R\Hom{A}{\widetilde{T}_{\rm{can}}}{-}}&&D^b(C(2,3,r+1)),}$$
sending $M\in\mmod \Gamma'$ to the (unique) simple projective $C(2,3,r+1)$-module.
Then by \cite{Lenzing:Pena:2011}, we obtain that $\Gamma\cong\Gamma'[M]$ has Fuchsian type $\langle 2,3,r+1]$.

This finishes the proof.

%
%
%
\end{proof}

\begin{remark} In order to show $\nak{r+6}{r}$ have Fuchsian type $\fuchs{2,3,r+1}$, it is natural to view $\nak{r+6}{r}$ as a one-point extension algebra $B[M]$ for some algebra $B$ and some $M\in\mmod B$, where $B$ is derived equivalent to the category of coherent sheaves over the weighted projective line $\XX$ of weight type $(2,3,r+1)$, and $M$ corresponds to a line bundle in $\coh\XX$.
In other words, we need to form a perpendicular category $E^{\perp}$ in $D^b(\nak{r+6}{r})$ with respect to an indecomposable object $E$, such that $E^{\perp}\simeq D^b(\coh\XX)$.

The first candidate is the perpendicular category of the last projective $\nak{r+6}{r}$-module $P_{r+6}$. However, by Corollary \ref{familis Nakayama corollary}, $P_{r+6}^{\perp}\simeq D^b(N_{r+5}(r))$ has Fuchsian type $\fuchs{2,3,r}$ for $r\geq 7$, which is not of sheaf type. From the proof of Proposition \ref{N-r+6-r} we see that, the perpendicular category $P_{r+5}^{\perp}\simeq I_{r+5}^{\perp}$ of the second last projective module $P_{r+5}$ has sheaf type $(2,3,r+1)$.

In particular, for $r=6$, i.e., in $D^b(\nak{12}{6})$ we have $P_{11}^{\perp}\simeq D^b(N_{11}(6))\simeq P_{12}^{\perp}$, which has sheaf type $(2,3,7)$.
\end{remark}
\section{Classification of Piecewise hereditary Nakayama categories}

In this section, we provide a new proof of the classification result for the Nakayama categories $\Nak{n}{r}:=D^b(\nak{n}{r})$ of piecewise hereditary type, which have been obtained by Happel and Seidel in \cite{Happel:Seidel:2010}.

Let us first introduce the extended version of the Happel-Seidel symmetry, which was first encountered in \cite{Happel:Seidel:2010} and plays an important role in the study of Nakayama algebras.

\begin{proposition}\label{HS-symmetry} Assume $a,b\geq 2$. Put $n=(a-1)(b-1)$. Then
 \begin{itemize}
 \item[(1)] $\Nak{n}{a}\simeq\Nak{n}{b}$;
 \item[(2)] $\Nak{n-1}{a}\simeq\Nak{n-1}{b}$;
 \item[(3)] $\Nak{n+1}{a}\simeq\Nak{n+1}{b}$.
 \end{itemize}
\end{proposition}

\begin{proof}
The first statement has been proved in \cite[Theorem 6.11]{Kussin:Lenzing:Meltzer:2013adv}, by observing that both categories $\Nak{n}{a}$ and $\Nak{n}{b}$ are equivalent to the stable category $\svect\X(2,a,b)$. For the proof of the other two statements, we recall from \cite{Kussin:Lenzing:Meltzer:2013adv} that
\begin{equation}
\label{tilting realization of N(n,a)}
T=\bigoplus\limits_{i=0}^{b-2}\bigoplus\limits_{j=0}^{a-2}\mathbb{S}^{i}(E\langle j\vx_2\rangle)
 \end{equation}is a tilting object in the stable category $\svect\X(2,a,b)$ with endomorphism algebra $N_{n}(a)$, where $\mathbb{S}=\tau[1]$, and $E\langle j\vx_2\rangle$ (for $0\leq j\leq a-2$) is the extension bundle determined by the non-split exact sequence
 $\xymatrix{0\ar[r]&\Oo(\vw)\ar[r]&E\langle j\vx_2\rangle\ar[r]&\Oo(j\vx_2) \ar[r]&0.}$ 

According to \cite[Proposition 6.6]{Kussin:Lenzing:Meltzer:2013adv}, we know that $E\langle (a-2)\vx_2\rangle$ is an Auslander bundle, and $\mathbb{S}(F)=\tau F[1]=F(\vom+\vx_1)$ for any extension bundle $F$. It follows that the last member $\mathbb{S}^{b-2}(E\langle (a-2)\vx_2\rangle)$ of $T$ is an Auslander bundle, hence $\Nak{n-1}{a}$ is equivalent to the right perpendicular category of an Auslander bundle in $\svect\X(2,a,b)$. By exchanging the weights $a$ and $b$ in $(2,a,b)$ (using $\svect\X(2,a,b)\simeq\svect\X(2,b,a)$), we know that $\Nak{n-1}{b}$ is equivalent to the right perpendicular category of $\mathbb{S}^{a-2}(E\langle (b-2)\vx_3\rangle)$, which is also an Auslander bundle. Note that there exists an automorphism (induced by degree shift) of $\svect\X(2,a,b)$ exchanging two given Auslander bundles. It follows that $\Nak{n-1}{a}\simeq\Nak{n-1}{b}$.

For the statement (3), recall that $N_{n+1}(a)=N_{n}(a)[I_{n+2-a}]$, which is derived equivalent to $N_{n}(a)[P_{n+2-a}]$ by \cite{Barot-Lenzing:2003}. Here, $I_{n+2-a}$ (resp. $P_{n+2-a}$) is the injective (resp. projective) $N_{n}(a)$-module corresponding to the vertex $n+2-a$. Moreover, by the explicit realization of $N_{n}(a)$ in $\svect\X(2,a,b)$ as in \eqref{tilting realization of N(n,a)}, $P_{n+2-a}$ is given by the Auslander bundle $\bS^{b-2}E$, (i.e., $j=0$). Similar result holds for $N_{n+1}(b)$, that is, $N_{n+1}(b)$ is derived equivalent to $N_{n}(b)[P_{n+2-b}]$, where
$P_{n+2-b}$ is given by the Auslander bundle $\bS^{a-2}E$.

Note that $\Nak{n}{a}\simeq \svect\X(2,a,b)\simeq\Nak{n}{b}$ by (1), and there exists an automorphism of $\svect\X(2,a,b)$ exchanging two given Auslander bundles. It follows that there exists an equivalence $\xymatrix{\Nak{n}{a}\ar[r]^-{\simeq}& \Nak{n}{b}
}$
sending projective $N_{n}(a)$-module $P_{n+2-a}$ to projective $N_{n}(b)$-module $P_{n+2-b}$.
Then by \cite{Barot-Lenzing:2003}, we obtain $N_{n+1}(a)$ and $N_{n+1}(b)$ are derived equivalent, that is, $\Nak{n+1}{a}\simeq\Nak{n+1}{b}$.
\end{proof}

\begin{lemma}\label{lemma:perp}
Assume that the Nakayama category $\Nak{n}{r}$ is piecewise hereditary. Then also $\Nak{n-1}{r}$ is piecewise hereditary.
\end{lemma}
\begin{proof}
The indecomposable projective modules $P_i$ $(1\leq i\leq n)$ over $\nak{n}{r}$ are exceptional objects in $\Nak{n}{r}$. Since the right perpendicular category of $\Nak{n}{r}$ with regard to the last projective $P_n$ is triangle equivalent to $\Nak{n-1}{r}$, the assertion follows from \cite{Geigle:Lenzing:1991} and \cite[Theorem 4.2]{Crawley-Boevey:2011}.
\end{proof}

\begin{proposition}\label{prop:fuchs}
A Fuchsian singularity category cannot be piecewise hereditary.
\end{proposition}
\begin{proof}

Following the pattern of \cite{Kussin:Lenzing:Meltzer:2013adv} one deduces that the singularity category of a Fuchsian singularity, as studied in \cite{Lenzing:Pena:2011}, is triangle equivalent to the stable category of vector bundles $\zvect\XX$ over a weighted projective line $\XX$.
Since, by \cite{Lenzing:Pena:1997}, all Auslander-Reiten components of $\vect\XX$ have type $\ZZ\AA_\infty$, the same assertion holds true for $\zvect\XX$. But this is not possible for a piecewise hereditary category: if it has sheaf type there will be AR-components that are tubes; if it has module type there will be components of type $\ZZ\Delta$ for a finite quiver $\Delta$.
\end{proof}

Now we can obtain the following classification result for piecewise hereditary Nakayama categories due to Happel and Seidel. Recall that the Nakayama category $\Nak{n}{2}$ is piecewise hereditary of type $\mathbb{A}_n$ for any $n$, while the Nakayama category $\Nak{n}{n-1}$ is piecewise hereditary of type $\mathbb{D}_n$ for $n\geq 4$, see \cite{Happel:Seidel:2010}.
We exclude these two cases in the following.

By abusing of notations, we use $\Nn\simeq (a,b,c)$ to denote that the Nakayama category $\Nn$ has sheaf type $(a,b,c)$, and similar for the other types.


\begin{theorem}\label{classification-of-piecewise-hereditary-nakayama-cat}\cite{Happel:Seidel:2010}
The complete classification of piecewise hereditary Nakayama categories $\Nn=\Nak{n}{r}$, for $n\geq5$ and $r\geq3$, is as follows:
\begin{enumerate}[label=(\roman*)]
\item[(1)] If $\Nn$ is piecewise hereditary of sheaf type, then one of the following holds:
\begin{enumerate}
\item[-] $\Nn$ belongs to the family $\Nak{r+4}{r}\, (r\geq 4)$, having sheaf type $\can{2,3,r}$;
\item[-] $\Nn$ is one of the categories $\Nak{9}{3}$, $\Nak{9}{6}$, $\Nak{9}{7}$, having sheaf type  $\can{2,3,5}$;
\item[-] $\Nn$ is one of the categories $\Nak{11}{3}, \Nak{11}{6}$, having sheaf type $\can{2,3,7}$;
 \item[-] $\Nn\simeq\Nak{10}{3}$, having sheaf type $\can{2,3,6}$;
 \item[-] $\Nn\simeq\Nak{9}{4}$, having sheaf type $\can{2,4,4}$;
 \item[-] $\Nn\simeq\Nak{10}{4}$, having sheaf type $\can{2,4,5}$.
\end{enumerate}
\item[(2)] If $\Nn$ is piecewise hereditary of module type, then one of the following holds:
 \begin{enumerate}
  \item[-] $\Nn$ belongs to the family $\Nak{r+2}{r}\, (r\geq 3)$, having module type $\hstar{2,3,r-1}$,
  \item[-] $\Nn$ belongs to the family $\Nak{r+3}{r}\, (r\geq 3)$, having module type $\hstar{2,3,r}$,
  \item[-] $\Nn$ is one of the categories $\Nak{9}{3}$, $\Nak{9}{5}$, having module type $\her{2,3,6}$,
  \item[-] $\Nn\simeq\Nak{7}{3}$, having module type $\her{2,3,4}$,
  \item[-] $\Nn\simeq\Nak{8}{3}$, having module type $\her{2,3,5}$,
  \item[-] $\Nn\simeq\Nak{8}{4}$, having module type $\her{2,4,4}$,
  \item[-] $\Nn\simeq\Nak{10}{5}$, having module type $\her{2,3,7}$.
 \end{enumerate}
\end{enumerate}
The overlap between sheaf type and module type are 
$(2,3,4)\simeq\Nak{8}{4}\simeq [2,4,4]$ and $$(2,3,5)\simeq\Nak{9}{3}\simeq\Nak{9}{5}\simeq\Nak{9}{6}\simeq\Nak{9}{7}\simeq [2,3,6].$$
\end{theorem}

\begin{proof}
By Proposition \ref{tilting realization} and Corollary \ref{familis Nakayama corollary}, we obtain the first three columns of piecewise hereditary Nakayama categories in the following figure (except for the member $\Nak{7}{3}$), and $\nak{r+5}{r}$ ($r\geq 7$) in the fourth column has Fuchsian type $\fuchs{2,3,r}$.

\small
\begin{figure}[H]
$$
\tiny
\def\l{\langle}
\def\y{\cellcolor{yellow}}
\def\r{\cellcolor{red}}
\def\g{\cellcolor{YellowGreen}}
\arraycolsep3.5pt
\begin {array}{c|cccccccccccccc}
r&\g\vdots  &\g\vdots  &\y\vdots  &\r\vdots&& & & & & & &   \\
&\g[2,3,18]&\g[2,3,19]&\y(2,3,19)&\r\l2,3,19]&& & & & & & &   \\ \noalign{\smallskip}
18&\g[2,3,17]&\g[2,3,18]&\y(2,3,18)&\r\l2,3,18]&& & & & & & &   \\ \noalign{\smallskip}
17&\g[2,3,16]&\g[2,3,17]&\y(2,3,17)&\r\l2,3,17]&& & & & & & &   \\ \noalign{\smallskip}
16&\g[2,3,15]&\g[2,3,16]&\y(2,3,16)&\r\l2,3,16]&& & & & & & &   \\ \noalign{\smallskip}
15&\g[2,3,14]&\g[2,3,15]&\y(2,3,15)&\r\l2,3,15]&& & & & & & &   \\ \noalign{\smallskip}
14&\g[2,3,13]&\g[2,3,14]&\y(2,3,14)&\r\l2,3,14]&& & & & & & &   \\ \noalign{\smallskip}
13&\g[2,3,12]&\g[2,3,13]&\y(2,3,13)&\r\l2,3,13]&& & & & & & &   \\ \noalign{\smallskip}
12&\g[2,3,11]&\g[2,3,12]&\y(2,3,12)&\r\l2,3,12]&& & & & & & &   \\ \noalign{\smallskip}
11&\g[2,3,10]&\g[2,3,11]&\y(2,3,11)&\r\l2,3,11]&& & & & & & &   \\ \noalign{\smallskip}
10&\g[2,3,9]&\g[2,3,10]&\y(2,3,10)&\r\l2,3,10]&& & & & & & &   \\ \noalign{\smallskip}
 9&\g[2,3,8]&\g[2,3,9]&\y(2,3,9)&\r\l2,3,9]&& & & & & & &   \\ \noalign{\smallskip}
 8&\g[2,3,7]&\g[2,3,8]&\y(2,3,8)&\r\l2,3,8]&&&& & & & &   \\ \noalign{\smallskip}
 7&\g[2,3,6]&\g[2,3,7]&\y(2,3,7)&\r\l2,3,7]&& & & & & & &   \\ \noalign{\smallskip}
 6&\g[2,3,5]&\g[2,3,6]&\y(2,3,6)&\y(2,3,7)&\r\l2,3,7]&&& & & & &   \\ \noalign{\smallskip}
5&\g[2,3,4]&\g[2,3,5]&\y(2,3,5)&\g[2,3,7]&\r\l2,4,5]&&&& & & &   \\ \noalign{\smallskip}
 4&\g[2,3,3]&\g[2,3,4]&\y(2,3,4)&\y(2,4,4)&\y(2,4,5)&\r\l2,4,5]&&&&  & &  \\ \noalign{\smallskip}
3&\g[2,3,2]&\g[2,3,3]&\g[2,3,4]&\g[2,3,5]&\y(2,3,5)&\y(2,3,6)&\y(2,3,7)&\r\l2,3,7]&&&& \\ \noalign{\smallskip}\hline
 &2& 3& 4&5&6 & 7&8 &9 && & & n-r\end {array}
$$
\caption{Module type (green), sheaf type (yellow) and the red wall, formed by Fuchsian singularities (red)}
\end{figure}
\normalsize

Now we show that the remaining sporadic Nakayama categories listed in the above figure have the desired types.

By Remark \ref{why r geq 7} and Proposition \ref{N-r+6-r}, we know that $$\Nak{10}{5}\simeq [2,3,7];\quad \Nak{11}{6}\simeq (2,3,7)  \text{\quad and\quad}  \Nak{12}{6}\simeq  \langle 2,3,7].$$
By Happel-Seidel symmetry, we have $\Nak{2(r-1)}{3}\simeq\Nak{2(r-1)}{r}$.
In particular, for $r=5,6,7$,
$$\Nak{8}{3}\simeq\Nak{8}{5}\simeq [2,3,5];\quad \Nak{10}{3}\simeq \Nak{10}{6}\simeq (2,3,6) \text{\quad and\quad} \Nak{12}{3}\simeq \Nak{12}{7}\simeq \langle2,3,7].$$
Then by Proposition \ref{HS-symmetry}, we obtain
$$\Nak{7}{3}\simeq\Nak{7}{5}\simeq [2,3,4];\quad \Nak{9}{3}\simeq \Nak{9}{6}\simeq (2,3,5) \text{\quad and\quad} \Nak{11}{3}\simeq   \Nak{11}{7}\simeq(2,3,7).$$
By Happel-Seidel symmetry again we have $\Nak{12}{4}\simeq\Nak{12}{5}$. It follows from Proposition \ref{HS-symmetry} and Theorem \ref{tiltingobjects} (proved independently in the next section) that $$\Nak{11}{4}\simeq\Nak{11}{5}\simeq\langle 2,4,5].$$
For $\X$ of tubular type (2,4,4), we have a triangle equivalence $\svect{\X}\simeq \Der{\X}$. Hence, $$\Nak{9}{4}\simeq \langle2,4,4\rangle\simeq (2,4,4).$$
Then by using one-point extension approach one easily obtains $$\Nak{10}{4}\simeq (2,4,5).$$

It is well known that there exists a derived equivalence between the category of coherent sheaves of a domestic weighted projective line and the representation category of the associated affine tame quiver. In particular, $$(2,3,4)\simeq [2,4,4] \text{\quad and\quad} (2,3,5)\simeq [2,3,6],$$ which have affine types $\widetilde{\mathbb{E}}_{7}$ and $\widetilde{\mathbb{E}}_{8}$, respectively.


In the above figure, the Nakayama categories $\Nak{12}{3}$, $\Nak{11}{4}$, $\Nak{11}{5}$, $\Nak{12}{6}$ and $\Nak{r+5}{r}$, $r\geq7$ form a ``red wall'', all having Fuchsian type. We show that to the right of the red wall, that is in the ``white area'', we have no further piecewise hereditary categories.
Indeed, assume that we have a piecewise hereditary category $\Nak{n}{r}$ belonging to the ``white area'', we thus find a category $\Nak{m}{r}$ of Fuchsian type with $m<n$. By means of Lemma~\ref{lemma:perp} $\Nak{m}{r}$ then must be piecewise hereditary, contradicting Proposition~\ref{prop:fuchs}.

The proof is complete.
\end{proof}

\section{Tilting objects in $\zvect\X$}

We assume $\chi_{\X}<0$ throughout this section, i.e., we are dealing with weighted projective lines $\X=\X(p_1, p_2, p_3)$ of wild type.

In order to classify all the Fuchsian Nakayama categories, we need to do further investigation on the stable category $\zvect\X$ of vector bundles in this section. After describing the projective covers and homomorphisms for line bundles, we finally construct tilting objects consisting of line bundles in $\zvect\X$ for certain $\X$, whose endomorphism algebras are Nakayama algebras.

Recall that $\mathcal{L}$ denotes the system of all line bundles in $\coh\X$. The number of $\tau$-orbits of $\mathcal{L}$ is given by $[\mathbb{L}:\mathbb{Z}\vec{\omega}]=-\chi_{\X}\cdot\prod_{i=1}^3 p_i$, see \cite[Lemma 4.19]{Lenzing:2011}. In order to give a set of representatives for the $\tau$-orbits of $\mathcal{L}$, we consider the following set
\begin{equation}\label{SS}\Ss:=\{\vx\, | \,0\leq \vx\leq n\vw+\vc \text{\ \ for\ any  }n\geq 2\}.
\end{equation}

\begin{proposition}\label{rep for line bundle orbits} There exists a bijection between the $\tau$-orbits of line bundles in $\coh\X$ and
$\Ss$.
\end{proposition}

\begin{proof}
  For any $\vx\in\LL$, by the assumption $\chi_{\X}<0$, i.e., $\delta(\vw)>0$, there exists a maximal integer $m=m(\vx)$, such that $\Hom{}{\co(m\vw)}{\co(\vx)}\neq 0$. Hence, $\vx-m\vw\geq 0$ and $\vx-(m+n)\vw\leq \vw+\vc$ for any $n\geq 1$. Thus, $\vx-m\vw\in\Ss$.
  Therefore, the following map from the $\tau$-orbits of line bundles $\mathcal{L}/\langle\tau\rangle$ to the set $\Ss$ is well-defined $$\varphi: \mathcal{L}/\langle\tau\rangle\to \Ss;\quad \co(\vx)\mapsto \vx-m(\vx)\vw.$$
  Obviously, $\varphi$ is surjective. Moreover, $\varphi(\co(\vx))=\varphi(\co(\vy))$ implies $\co(\vx)$ and $\co(\vy)$ belong to the same $\tau$-orbit, hence $\varphi$ is injective.
  This finishes the proof.
\end{proof}

\begin{lemma}\label{n omega positive} Assume $a\vx_i\in\Ss$ for some $a>0$ and $1\leq i\leq 3$. Then  $(p_i-a+j)\vw>0$ for any $1\leq j\leq a$.
\end{lemma}

\begin{proof}
 Since $a\vx_i\leq 2\vw+\vc=\sum_{1\leq j\leq 3}(p_j-2)\vx_j$, we have $a\leq p_i-2$. Thus for any $1\leq j\leq a$, $p_i-a+j> p_i-a\geq 2$. Hence by assumption, $a\vx_i\leq (p_i-a+j)\vw+\vc=(a-j)\vx_i+\vy$ for some $\vy\in\bl$ without $\vx_i$ term in its normal form. Thus we have $\vy\geq \vc$ and then $(p_i-a+j)\vw> 0$ follows.
\end{proof}

For convenience, we denote by $\cp(E)$ the projective cover of a vector bundle $E\in\vect\X$ under the $\tau^\mathbb{Z}\co$-exact structure on $\vect\X$. We have the following description on projective covers for certain line bundles.

\begin{proposition}\label{projective cover}  Let $\{i,j,k\}=\{1,2,3\}$. The following statements hold.
 \begin{enumerate}
  \item If $\vx_i\in\Ss$, then $\cp(\co(\vx_i))=\co\oplus\co(-\vw-p_i\vw)$.
  \item If $\vx_i+\vx_j\in\Ss$, then $\cp(\co(\vx_i+\vx_j))=\co\oplus\co(-\vw)$.
  \item If $2\vx_i\in\Ss$, then $\cp(\co(2\vx_i))=\co\oplus\co(-2\vw)$.
  \end{enumerate}
\end{proposition}

\begin{proof}
  (1) By Lemma \ref{n omega positive}, $\vx_i\in\Ss$ implies that $p_i\vw>0$. Note that the coefficient of $\vx_i$ in the normal form of $p_i\vw$ is zero. Hence $\vx_i+\vw+mp_i\vw\leq \vw+\vc$ if and only if $mp_i\vw\leq 0$, if and only if $m\leq 0$. Then $\Hom{}{\co(-\vw-m p_i\vw)}{\co(\vx_i)}=S_{\vx_i+\vw+mp_i\vw}\neq 0$ if and only if $m\geq 1$.
   Write $\vx_i+\vw+p_i\vw=a_j\vx_j+a_k\vx_k+a\vc$ in normal form and let $\phi=x_j^{a_j+ap_j}x_k^{a_k}$. We claim that the homomorphism $$\xymatrix{(x_i, \phi): \co\oplus \co(-\vw-p_i\vw)\ar[r]&\co(\vx_i)}$$ gives the projective cover of $\co(\vx_i)$.

  In fact, for any line bundle $L$ from the orbit $\tau^{\ZZ}\co$ and any  $0\neq \psi\in\Hom{}{L}{\co(\vx_i)}$, it suffices to show $\psi$ factors through $(x_i, \phi)$.
We consider the following exact sequence
  $$\xymatrix{0\ar[r]&\co\ar[r]^-{x_i}&\co(\vx_i)\ar[r]^-{\pi_i}& S_{i1}\ar[r]& 0.}$$
  If $\pi_i\psi=0$, then $\psi$ factors through $x_i$ already. If else, $\Hom{}{L}{S_{i1}}\neq 0$ implies $L$ has the form $\co(-\vw-m p_i\vw)$ for some $m\geq 1$. We write $\vx_i+\vw+mp_i\vw=a'_j\vx_j+a'_k\vx_k+a'\vc$ in normal form. Then $\psi\in S_{\vx_i+\vw+mp_i\vw}$ has the form $\psi=x_j^{a'_j}x_k^{a'_k}f_{a'}(x_i^{p_i}, x_j^{p_j})$, where $f_{a'}(T_i, T_j)$ is a $\ZZ$-homogeneous polynomial of degree $a'$ on variables $T_i$ and $T_j$. We can write $\psi$ as $\psi=x_i^{p_i}\psi_1+\psi_2$ such that $\psi_2$ is not divided by $x_i^{p_i}$, that is, $\psi_2=x_j^{a'_j+a'p_j}x_k^{a'_k}$. Recall that $p_i\vw>0$. It follows that $a_j\vx_j+a_k\vx_k+a\vc\leq a'_j\vx_j+a'_k\vx_k+a'\vc$. Therefore, we reduce to the following two cases:
   \begin{enumerate}
       \item[-] if $a_k\leq a'_k$, then $a_j\vx_j+a\vc\leq a'_j\vx_j+a'\vc$, hence $a_j+ap_j\leq a'_j+a'p_j$, thus $\psi_{2}$ factors through $\phi$;
       \item[-] if $a_k> a'_k$, then $a_j\vx_j+a\vc\leq a'_j\vx_j+(a'-1)\vc$, hence $a_j+ap_j\leq a'_j+(a'-1)p_j$, thus $\psi_{2}=x_j^{a'_j+(a'-1)p_j}x_k^{a'_k}(-x_i^{p_i}-x_k^{p_k})$, which factors through $(x_i, \phi)$.
    \end{enumerate}
    To sum up, in both cases above, $\psi=x_i^{p_i}\psi_1+\psi_2$ factors through $(x_i, \phi)$.

    (2) Since $\vx_i+\vx_j\in\Ss$, we have $\vx_j\in\Ss$ and then $p_j\vw>0$ by Lemma
    \ref{n omega positive}. Hence, $\Hom{}{\co(-\vw-m p_j\vw)}{\co(\vx_i+\vx_j)}=S_{(p_k-1)\vx_k+mp_j\vw}\neq 0$ if and only if $m\geq 0$. We claim that the homomorphism
    $$\xymatrix{(x_ix_j, x_k^{p_k-1}): \co\oplus \co(-\vw)\ar[r]&\co(\vx_i+\vx_j)}$$ gives the projective cover of $\co(\vx_i+\vx_j)$.

    In fact, for any line bundle $L$ from the orbit $\tau^{\ZZ}\co$ and any $0\neq\psi\in\Hom{}{L}{\co(\vx_i+\vx_j)}$, it suffices to show $\psi$ factors through $(x_ix_j, x_k^{p_k-1})$. We consider the following exact sequence  $$\xymatrix{0\ar[r]&\co(\vx_i)\ar[r]^-{x_j}&\co(\vx_i+\vx_j)\ar[r]^-{\pi_j}& S_{j1}\ar[r]& 0.}$$ If $\pi_j\psi=0$, then $\psi$ factors through $x_j$. If else, $\Hom{}{L}{S_{j1}}\neq 0$ implies $L$ has the form $\co(-\vw-m p_j\vw)$ for some $m\geq 0$, and $\psi\in S_{(p_k-1)\vx_k+mp_j\vw}$. Recall that $p_j\vw>0$, it follows that $\psi$ factors through $(x_j, x_k^{p_k-1})$. In both cases, we know that $\psi$ factors through $(x_jx_i, x_j\phi, x_k^{p_k-1})$ by using (1). Hence, to finish the claim, we only need to show that $x_j\phi$ factors through $(x_ix_j, x_k^{p_k-1})$.

    Recall that $\phi=x_j^{ap_j+a_j}x_k^{a_k}$ and $\vx_i+\vw+p_i\vw=a_j\vx_j+a_k\vx_k+a\vc$. If $a>0$, then $x_j\phi=x_j^{ap_j+a_j+1}x_k^{a_k}=x_j^{(a-1)p_j+a_j+1}x_k^{a_k}(-x_i^{p_i}-x_k^{p_k})$, which factors through $(x_ix_j, x_k^{p_k-1})$. If $a=0$, we claim that $a_k=p_k-1$, then it follows that $x_j\phi$ factors through $x_k^{p_k-1}$. In fact, since $p_i\vw>0$, we have $\vx_i+\vx_j+\vw+ p_i\vw=(a_j+1)\vx_j+a_k\vx_k\geq(p_k-1)\vx_k$. If $a_j\leq p_j-2$, then we are done; if $a_j= p_j-1$, then $\vx_i+\vx_j+\vw+ p_i\vw=\vc+a_k\vx_k$, hence $p_i\vw=(a_k+1)\vx_k$. Recall that $\vx_i+\vx_j\in\Ss$. Hence $\vx_i+\vx_j\leq p_i\vw+\vc=(a_k+1)\vx_k+\vc$, which implies $a_k=p_k-1$.

    (3) By Lemma \ref{n omega positive}, $2\vx_i\in\Ss$ implies that $p_i\vw>0$ and $(p_i-1)\vw>0$. Similarly as before, we have $\Hom{}{\co(-2\vw-m p_i\vw)}{\co(2\vx_i)}=S_{(p_j-2)\vx_j+(p_k-2)\vx_k+mp_i\vw}\neq 0$ if and only if $m\geq 0$. We claim that the homomorphism $$\xymatrix{(x_i^2, x_j^{p_j-2}x_k^{p_k-2}): \co\oplus \co(-2\vw)\ar[r]&\co(2\vx_i)}$$ gives the projective cover of $\co(2\vx_i)$.

    In fact, for any line bundle $L$ from the orbit $\tau^{\ZZ}\co$ and any  $0\neq\psi\in\Hom{}{L}{\co(2\vx_i)}$, it suffices to show $\psi$ factors through
    $(x_i^2, x_j^{p_j-2}x_k^{p_k-2})$. We consider the following exact sequence  $$\xymatrix{0\ar[r]&\co(\vx_i)\ar[r]^-{x_i}&\co(2\vx_i)\ar[r]^-{\pi_i}& S_{i2}\ar[r]& 0.}$$ If $\pi_i\psi=0$, then $\psi$ factors through $x_i$. If else, $\Hom{}{L}{S_{i2}}\neq 0$ implies $L$ has the form $\co(-2\vw-m p_i\vw)$ for some $m\geq 0$, and $\psi\in S_{(p_j-2)\vx_j+(p_k-2)\vx_k+mp_i\vw}$.
    Recall that $p_i\vw>0$, it follows that $\psi$ factors through
    $(x_i, x_j^{p_j-2}x_k^{p_k-2})$. In both cases, $\psi$ factors through
    $(x_i^2, x_i\phi, x_j^{p_j-2}x_k^{p_k-2})$ by statement (1). Hence, to finish the claim, we only need to show that $x_i\phi$ factors through $(x_i^2, x_j^{p_j-2}x_k^{p_k-2})$, which
    follows from the fact $(p_i-1)\vw>0$ by using similar arguments as in (2).

    This finishes the proof.
\end{proof}

Denote the homomorphism space in $\zvect\X$ by $\sHom(E,F)$ for $E,F\in\vect\X$.

\begin{proposition}\label{non-zero morphisms between line bundles} Let $\{i,j,k\}=\{1,2,3\}$. Assume $\vx\geq 0$ has normal form $\vx=\sum_{1\leq i\leq 3}l_i\vx_i+l\vc$. Then the following hold.
 \begin{enumerate}
  \item If $\vx_i\in\Ss$, then $\sHom(\co(\vx_i-\vx), \co(\vx_i))\neq 0$ if and only if $l_i=0$ and $\vx\ngeq \vx_i+\vw+p_i\vw$.
  \item If $\vx_i+\vx_j\in\Ss$, then $\sHom(\co(\vx_i+\vx_j-\vx), \co(\vx_i+\vx_j))\neq 0$ if and only if $l_k\leq p_k-2$ and $\vx\ngeq \vx_i+\vx_j$.
  \item If $2\vx_i\in\Ss$, then $\sHom(\co(2\vx_i-\vx), \co(2\vx_i))\neq 0$ if and only if $l_i\leq 1$ and $\vx\ngeq (p_j-2)\vx_j+(p_k-2)\vx_k$.
  \end{enumerate}
Moreover, in each case, the nonzero homomorphism space has dimension one.
\end{proposition}

\begin{proof}
  We only prove the first statement, since the others are analogous. Recall from Proposition \ref{projective cover} that $(x_i, \phi): \co\oplus \co(-\vw-mp_i\vw)\to\co(\vx_i)$ gives the projective cover of $\co(\vx_i)$. Hence, an element $\psi\in\Hom{}{\co(\vx_i-\vx)}{\co(\vx_i)}$ does not vanish in the stable category $\zvect\X$ if and only if $\psi$ can not factor through $(x_i, \phi)$. Then the result follows.
\end{proof}

%

Now we can state the main result of this section.

\begin{theorem}\label{tiltingobjects}
\begin{enumerate}
  \item In $\zvect\X(2,4,5)$, $$T_{(2,4,5)}=\bigoplus\limits_{k=0}^{10}\co(\vx_2+k\vx_3)$$
       is a tilting object whose endomorphism algebra is isomorphic to $N_{11}(5)$.
  \item In $\zvect\X(2,4,7)$, $$T_{(2,4,7)}=\big(\bigoplus\limits_{k=0}^{6}\co(\vx_2+3k\vx_3)\big)\oplus
      \big(\bigoplus\limits_{k=0}^{5}\co(\vx_2+(3k+1)\vx_3)\big)$$
      is a tilting object whose endomorphism algebra is isomorphic to $N_{13}(6)$.
  \item In $\zvect\X(2,5,5)$, $$T_{(2,5,5)}=\bigoplus\limits_{k=0}^{2}\bigoplus\limits_{a=1}^{4}\bS^k(\co(a\vx_3))$$
 is a tilting object whose endomorphism algebra is isomorphic to $N_{12}(5)$.
  \item In $\zvect\X(2,5,6)$,
  $$T_{(2,5,6)}=\big(\bigoplus\limits_{k=0}^{3}\bS^k(\co(\vx_3))\big)\oplus \big( \bigoplus\limits_{k=0}^{2}\bS^k(\co(2\vx_3)\oplus \co(4\vx_3)\oplus \co(6\vx_3))\big)$$
  is a tilting object whose endomorphism algebra is isomorphic to  $N_{13}(5)$.
\end{enumerate}
\end{theorem}


%


\begin{proof} First we need to show that all of $T_{(2,4,5)}, T_{(2,4,7)}, T_{(2,5,5)}$ and $T_{(2,5,6)}$ are extension-free, which have been done case by case. In order to make the paper readable, we put them in the appendix, see Propositions \ref{extension-free 245}, \ref{extension-free 247}, \ref{extension-free 255} and \ref{extension-free 256} respectively.

Secondly, observe that for $T=T_{(2,4,5)}, T_{(2,4,7)}, T_{(2,5,5)}$ or $T_{(2,5,6)}$, the indecomposable direct summands of $T$ are line bundles up to suspension shift, which can be ordered to form an exceptional sequence. Thus the smallest triangulated subcategory $\Cc$ of $\zvect\X$ containing $T$ is generated by an exceptional sequence. By \cite{Bondal:Kapranov:1989} $\zvect\X$ is generated by $\Cc$ together with $\Cc ^{\bot}$. Since the number of indecomposable summands of $T$ agrees with the rank of the Grothendieck group of $\zvect\X$, which equals to $p_1+p_2+p_3$ for $\X=\X(p_1,p_2,p_3)$. It follows that $\Cc ^{\bot}=0$, hence $T$ generates $\Cc$. Therefore, $T$ is a tilting object in $\zvect\X$.

Finally, by using Proposition \ref{non-zero morphisms between line bundles}, it is easy to see that
the endomorphism algebra of the tilting object $T$ is isomorphic to the desired Nakayama algebra in each case.
This finishes the proof of Theorem \ref{tiltingobjects}.
\end{proof}

\begin{remark}
  (1) For $\X$ of weight type (2,4,5), we have $4\vw=\vx_3$. Let $L=\co(\vx_2)$, then $L(k\vx_3)\, (k\in\mathbb{Z})$ belong to the same $\tau$-orbit. Moreover, $\tau L[1]=L(16\vw)=L(4\vx_3)$, see Appendix \ref{subsection (245)}.
Hence $T_{(2,4,5)}$ can be written as $$T_{(2,4,5)}=\big(\bigoplus\limits_{k=0}^{2}\bigoplus\limits_{a=0}^{2}\bS^k(L(a\vx_3))\big)\oplus \big( \bigoplus\limits_{k=0}^{1}\bS^k(L(3\vx_3))\big). $$

  (2) For $\X$ of weight type (2,4,7), we have $4\vw=3\vx_3$. Let $L=\co(\vx_2)$, then $L(3k\vx_3)\, (k\in\mathbb{Z})$ belong to the same $\tau$-orbit. Moreover, $\tau L[1]=L(7\vx_3)$ and $\tau L(\vx_3)[1]=L(9\vx_3)$, see Appendix \ref{subsection (247)}.
Hence $T_{(2,4,7)}$ can be written as $$T_{(2,4,7)}=\big(\bigoplus\limits_{k=0}^{2}\bS^k(L\oplus L(\vx_3)\oplus L(3\vx_3))\big)\oplus \big( \bigoplus\limits_{k=0}^{1}\bS^k(L(4\vx_3)\oplus L(6\vx_3))\big). $$
\end{remark}

\begin{conjecture} \label{conj:257 and 266 and 2311(derived version)} We expect that the following statements hold:
\begin{enumerate}
  \item $N_{14}(4)$ has Fuchsian type $\fuchs{2,5,7}$;
  \item $N_{14}(5)$ has Fuchsian type $\fuchs{2,6,6}$;
  \item $N_{16}(8)$ has Fuchsian type $\fuchs{2,3,11}$.
\end{enumerate}
\end{conjecture}

For the weight types appearing in Theorem \ref{tiltingobjects}, the associated Fuchsian singularities are all hypersurfaces, hence the second suspension functor [2] of $\zvect\X$ are given by certain degree shifts by using matrix factorization approach. This observation plays an important role during our proof for an object being extension-free in $\zvect\X$. However, for the weight types appearing in Conjecture \ref{conj:257 and 266 and 2311(derived version)}, the associated Fuchsian singularities are complete intersections but not hypersurfaces, hence the suspension functors of $\zvect\X$ do not have nice description sofar, which makes it much more difficult to find a tilting object in $\zvect\X$ whose endomorphism algebra is isomorphic to the required Nakayama algebra.


Nevertheless,
There is an important evidence to support Conjecture \ref{conj:257 and 266 and 2311(derived version)}. In fact, for each of the weight types $(2,5,7)$, $(2,6,6)$ or $(2,3,11)$, the stable category $\zvect\X$ (or equivalently, the related extended canonical algebra) and the associated Nakayama algebra share the same Coxeter polynomial. We believe that the Coxeter polynomials are complete invariants under derived equivalences for certain classes of algebras. In fact, we have the following conjecture.

\begin{conjecture} \label{conj:Coxeter polynomial and derived equivalence} Two Nakayama algebras $N_{n}(a)$ and $N_{n}(b)$ are derived equivalent if and only if they share the same Coxeter polynomial.
\end{conjecture}
Conjecture \ref{conj:Coxeter polynomial and derived equivalence} holds true between Nakayama algebras of module type, of sheaf type or of triangle type, see \cite{Kussin:Lenzing:Meltzer:2013adv} and \cite{Hille:Mueller:2014}.

\section{Classification of Fuchsian Nakayama categories}

In this section, we will give the complete classification of
Nakayama categories $\Nak{n}{r}:=D^b(N_{n}(r))$ of Fuchsian type.

First let us consider the Coxeter polynomials for a class of Nakayama algebras $N_{r+7}(r)$, which are of interest in their own right.

\begin{proposition}\label{Coxeter numbers}
For any $r\geq 9$, the Coxeter polynomial of $N_{r+7}(r)$ is given by $$\chi_{(r+7,r)}=(\lambda+1)(\lambda^6-\lambda^3+1)(\lambda^r+1).$$
Consequently, the Coxeter number of $N_{r+7}(r)$ is given by $\lcm(2r,9)$.
\end{proposition}

\begin{proof}
Denote the Cartan matrix of $N_{r+7}(r)$ by $C$. Then by direct computation we have
$$C^{-t}C=\begin{bmatrix}O_{6\times8}&\begin{pmatrix}O_{6\times(r-7)}&-E_6\end{pmatrix}\\
O_{1\times8}&\varepsilon_{r-1}\\
\begin{pmatrix}O_{(r-8)\times8}\\
E_8\end{pmatrix}&\begin{pmatrix}-E_{r-1}\\
\varepsilon_{r-1}\end{pmatrix}\end{bmatrix},$$
where $E_n$ denotes the identity matrix of size $n\times n$, $O_{m\times n}$ denotes the zero matrix of size $m\times n$, and $\varepsilon_{r-1}$ denotes the row vector $(1,1,1,\cdots,1)$ with $(r-1)$-entries.
By definition, the Coxeter polynomial $\chi_{(r+7,r)}$ of $N_{r+7}(r)$ is the determinant of the matrix $$\Delta:=\lambda E_{n+r}-\Phi_{(n+r,r)}=\lambda E_{n+r}+C^{-t}C.$$

Let ${\bf{r}}=(9,10,\cdots,r-1)$ be a list of $(r-9)$-row indices and ${\bf{c}}_i=(9,10,\cdots,\hat{i},\cdots,r-1,r) \; (9\leq i\leq r)$ be lists of $(r-9)$-column indices of $\Delta$.
For any $9\leq i\leq r$, the submatrix $\Delta({\bf{r}},{\bf{c}}_i)$, obtained by keeping the entries in the intersection of any row in the list ${\bf{r}}$ and any column in the list ${\bf{c}}_i$, has the form $$\Delta({\bf{r}},{\bf{c}}_i)=\begin{pmatrix}J_i&0
\\0&J_{i}'\end{pmatrix},$$
where $J_i$ is an $(i-9)\times(i-9)$ matrix of the following form:
$$J_i=\begin{pmatrix}\lambda&-1&0&\cdots&0\\
0&\lambda&-1&\ddots&\vdots\\
\vdots&\ddots&\ddots&\ddots&0\\
\vdots&&\ddots&\ddots&-1\\
0&\cdots&\cdots&0&\lambda\end{pmatrix},$$
and $J_i'$ is an $(r-i)\times(r-i)$ matrix of the following form:
 $$J_i'=\begin{pmatrix}-1&0&0&\cdots&0\\
\lambda&-1&\ddots&&\vdots\\
0&\ddots&\ddots&\ddots&\vdots\\
\vdots&\ddots&\ddots&\ddots&0\\
0&\cdots&0&\lambda&-1\end{pmatrix}.$$
Hence, the determinant $$\textup{det}(\Delta({\bf{r}},{\bf{c}}_i))=\textup{det}(J_i)\cdot\textup{det}(J_i')=(-1)^{r-i}\lambda^{i-9}.$$
The complementary submatrix $\Delta'({\bf{r}},{\bf{c}}_i)$, obtained by removing the entries in the rows in the list ${\bf{r}}$ and columns in the list ${\bf{c}}_i$, can be calculated as follows according to $9\leq i\leq r$.

For $i=9$, we have $$\Delta'({\bf{r}},{\bf{c}}_i)=\begin{pmatrix}\lambda E_8&M\\E_8&N\end{pmatrix},$$ where
 $$M=\begin{pmatrix}0&0&-1&0&0&0&0&0\\
0&0&0&-1&0&0&0&0\\
0&0&0&0&-1&0&0&0\\
0&0&0&0&0&-1&0&0\\
0&0&0&0&0&0&-1&0\\0&0&0&0&0&0&0&-1\\1&1&1&1&1&1&1&1\\
-1&0&0&0&0&0&0&0\end{pmatrix}$$ and $$N=\begin{pmatrix}0&-1&0&0&0&0&0&0\\
0&\lambda&-1&0&0&0&0&0\\
0&0&\lambda&-1&0&0&0&0\\
0&0&0&\lambda&-1&0&0&0\\0&0&0&0&\lambda&-1&0&0\\0&0&0&0&0&\lambda&-1&0\\
0&0&0&0&0&0&\lambda&-1\\
1&1&1&1&1&1&1&\lambda+1\end{pmatrix}.$$
Hence $$\textup{det}(\Delta'({\bf{r}},{\bf{c}}_9))=\begin{vmatrix}\lambda E_8&M\\E_8&N\end{vmatrix}=\begin{vmatrix}\lambda N-M\end{vmatrix}=(\lambda+1)(\lambda^6-\lambda^3+1).$$
For $10\leq i\leq r-1$, we have $$\Delta'({\bf{r}},{\bf{c}}_i)=\begin{pmatrix}\lambda E_8&M+E_{8,1}\\E_8&N\end{pmatrix},$$ where $E_{i,j}$ denotes the elementary $8\times 8$ matrix with a unique nonzero entry 1 in the $(i,j)$-position. By a direct calculation we get
$$\textup{det}(\Delta'({\bf{r}},{\bf{c}}_i))=
\begin{vmatrix}\lambda N-(M+E_{8,1})\end{vmatrix}=0.$$
For $i=r$, we have $$\Delta'({\bf{r}},{\bf{c}}_r)=\begin{pmatrix}\lambda E_8&M+E_{8,1}\\E_8&N+\lambda E_{1,1}\end{pmatrix}.$$ Hence
$$\textup{det}(\Delta'({\bf{r}},{\bf{c}}_r))
=\begin{vmatrix}\lambda (N+\lambda E_{1,1})-(M+E_{8,1})\end{vmatrix}=\lambda^9(\lambda+1)(\lambda^6-\lambda^3+1).$$


Therefore, by Laplace Expansion Theorem we obtain \begin{align}\textup{det}(\Delta)&=(-1)^{9+10+\cdots+(r-1)}\sum_{i=9}^r
(-1)^{9+10+\cdots+\hat{i}+\cdots+r}\cdot\textup{det}(\Delta({\bf{r}},{\bf{c}}_i))\cdot
\textup{det}(\Delta'({\bf{r}},{\bf{c}}_i))
\nonumber\\
&=\sum_{i=9}^r(-1)^{r+i}\cdot(-1)^{r-i}\lambda^{i-9}\cdot\textup{det}(\Delta'({\bf{r}},{\bf{c}}_i))\nonumber\\
&=(\lambda+1)(\lambda^6-\lambda^3+1)+\lambda^{r-9}\cdot\lambda^9(\lambda+1)(\lambda^6-\lambda^3+1).\nonumber\\
&=(\lambda+1)(\lambda^6-\lambda^3+1)(\lambda^r+1).\nonumber\end{align}

Note that each root $\xi$ of the polynomial $\lambda^6-\lambda^3+1$ satisfies $\xi^9+1=0$, hence $\xi$ has order 18. Similarly, each root of $\lambda^r+1$ has order $2r$. It follows that the Coxeter number of $N_{r+7}(r)$ is given by $\lcm(2, 18, 2r)=\lcm(9, 2r)$.
\end{proof}

In order to classify the Fushsian Nakayama categories, we state the following conjecture.

\begin{conjecture} \label{conj:Fuchsian:perp} We expect that each of the following two properties is satisfied.
\begin{itemize}
\item[(F1)] Each connected perpendicular subcategory of a Fuchsian singularity category is either Fuchsian or piecewise hereditary.
\item[(F2)] If a Nakayama category $\Nak{n}{r}$ is Fuchsian, then $\Nak{n-1}{r}$ is either Fuchsian or piecewise hereditary.
\end{itemize}
\end{conjecture}
Clearly $(F1)$ implies $(F2)$. In support of Conjecture \ref{conj:Fuchsian:perp}, we first note each Fuchsian singularity category has an orthogonal decomposition $\langle\Der{\coh\XX},E\rangle$ for some weighted projective line $\X$ and an exceptional object $E$, as follows from \cite{Lenzing:Pena:2011}. It is thus conceivable that $(F1)$ can be derived from this property using a descent argument. Second, the experimental results collected in Figure~\ref{figure:green:wall} show that Fuchsian Nakayama categories appear in intervals ending --- reading from right to left --- in piecewise hereditary categories, actually in such categories of sheaf type. Indeed, on an experimental level, $(F2)$ is satisfied by all Nakayama categories $\Nak{s+r}{r}$ in the range $3\leq r\leq 50$, $2\leq s\leq 50$.

Under Conjecture \ref{conj:Fuchsian:perp} and Conjecture \ref{conj:257 and 266 and 2311(derived version)}, we obtain the main result of this paper as follows, see also Figure~\ref{figure:green:wall}.

\begin{theorem}\label{classification-of-Fuchsian-nakayama-cat}
The complete classification of Fuchsian Nakayama categories $\Nn=\Nak{n}{r}$, for $n\geq5$ and $r\geq3$, is as follows:
\begin{enumerate}[label=(\roman*)]
\item[(1)] $\Nn$ belongs to the following two families:
  \begin{enumerate}
  \item[-] $\Nak{r+5}{r}, r\geq 7$, having Fuchsian type $\langle 2,3,r];$
  \item[-] $\Nak{r+6}{r}, r\geq 6$, having Fuchsian type $\langle 2,3,r+1].$
  \end{enumerate}
\item[(2)] $\Nn$ is one of the hypersurface singularity categories:
 \begin{enumerate}
  \item[-] $\Nak{r}{3}$, $r=12, 13, 14$, having Fuchsian type $\langle 2,3,r-5];$
  \item[-] $\Nak{r}{4}$ or $\Nak{r}{5}$, $r=11, 12, 13$, having Fuchsian type $\langle 2, 5, r-7];$
  \item[-] $\Nak{13}{6}$, having Fuchsian type $\langle 2, 4, 7].$
 \end{enumerate}
\item[(3)] $\Nn$ is one of the complete intersection singularity categories:
 \begin{enumerate}
  \item[-] $\Nak{15}{3}$ or $\Nak{15}{8}$, having Fuchsian type $\langle 2, 3, 10];$
  \item[-] $\Nak{14}{4}$ or $\Nak{14}{6}$, having Fuchsian type $\langle 2, 5, 7];$
  \item[-] $\Nak{14}{5}$, having Fuchsian type $\langle 2, 6, 6];$
  \item[-] $\Nak{16}{8}$, having Fuchsian type $\langle 2, 3, 11].$
  \end{enumerate}
\end{enumerate}
\end{theorem}

\begin{proof}
Firstly, we show that the Nakayama categories involved in the ``red'' part of Figure~\ref{figure:green:wall} have desired Fuchsian types.
The first part has been proved in Corollary \ref{familis Nakayama corollary} (3) and Proposition \ref{N-r+6-r}.
 We now prove part (2) case by case.

 For the following three weight types, we use the (Extended) Happel-Seidel symmetry, see Proposition \ref{HS-symmetry}.
    \begin{itemize}
     \item[-] For $\X$ of weight type (2,3,7), there has a unique line bundle orbit in the category $\vect\X(2,3,7)$. Hence $\tau^{\ZZ}\co$-exact structure coincides with the $\mathcal{L}$-exact structure on $\vect\X(2,3,7)$.
         Thus $\langle 2,3,7]\simeq \langle 2,3,7\rangle\simeq \Nak{12}{3}$, see \cite{Kussin:Lenzing:Meltzer:2013adv}.
     \item[-] For $\X$ of weight type (2,3,8), we have $\langle 2,3,8]\simeq \Nak{13}{7}$ by part (1). Note that we have the Happel-Seidel symmetry $\Nak{12}{3}\simeq \langle 2,3,7\rangle\simeq \Nak{12}{7}$. It follows from Proposition \ref{HS-symmetry} (3) that $\Nak{13}{3}\simeq \Nak{13}{7}\simeq \langle 2,3,8]$.
     \item[-] For $\X$ of weight type (2,3,9), we have $\langle 2,3,9]\simeq \Nak{14}{8}$ by part (1). By Happel-Seidel symmetry $\Nak{14}{3}\simeq \langle 2,3,8\rangle\simeq \Nak{14}{8}\simeq \langle 2,3,9]$.
          \end{itemize}

The following four weight types rely on the realization of Nakayama algebras by tilting objects in the stable category $\zvect\X$, see Theorem \ref{tiltingobjects}.
            \begin{itemize}
            \item[-] For $\X$ of weight type (2,4,5), by Theorem \ref{tiltingobjects} (1) there exists a tilting object in $\langle 2,4,5]$ with endomorphism algebra $N_{11}(5)$. Thus $\langle 2,4,5]\simeq  \Nak{11}{5}$. Moreover, we have the Happel-Seidel symmetry $\Nak{12}{4}\simeq \langle 2,4,5\rangle \simeq \Nak{12}{5}$. Thus $\Nak{11}{4}\simeq \Nak{11}{5}\simeq\langle 2,4,5]$ by Proposition \ref{HS-symmetry} (2).
         \item[-] For $\X$ of weight type (2,5,5), by Theorem \ref{tiltingobjects} (2) there exists a tilting object in $\langle 2,5,5]$ with endomorphism algebra $N_{12}(5)$. Thus $\langle 2,5,5]\simeq  \Nak{12}{5}\simeq\langle 2,4,5\rangle\simeq \Nak{12}{4}$ by Happel-Seidel symmetry.
         \item[-] For $\X$ of weight type (2,5,6), by Theorem \ref{tiltingobjects} (3) there exists a tilting object in $\langle 2,5,6]$ with endomorphism algebra $N_{13}(5)$. Thus $\langle 2,5,6]\simeq  \Nak{13}{5}\simeq \Nak{13}{4}$ by Proposition \ref{HS-symmetry} (3).

        \item[-] For $\X$ of weight type (2,4,7), by Theorem \ref{tiltingobjects} (4) there exists a tilting object in $\langle 2,4,7]$ with endomorphism algebra $N_{13}(6)$. Thus $\langle 2,4,7]\simeq  \Nak{13}{6}$.
      \end{itemize}

 Now we turn to consider part (3), where Proposition \ref{HS-symmetry} plays an important role.
    \begin{itemize}
     \item[-] For $\X$ of weight type (2,3,10), we have $\langle 2,3,10]\simeq \Nak{15}{9}$ by part (1). Note that we have the Happel-Seidel symmetry $\Nak{16}{3}\simeq \langle 2,3,9\rangle\simeq \Nak{16}{9}$. It follows that $\Nak{15}{3}\simeq \Nak{15}{9}$. Moreover, by Happel-Seidel symmetry $\Nak{14}{3}\simeq \langle 2,3,8\rangle\simeq \Nak{14}{8}$, we finally get $\Nak{15}{8}\simeq \Nak{15}{3}\simeq \Nak{15}{9}\simeq \langle 2,3,10]$.
   \item[-] For $\X$ of weight type (2,5,7), under Conjecture
     \ref{conj:257 and 266 and 2311(derived version)} we have $\langle 2,5,7]\simeq  \Nak{14}{4}$.
     Note that we have the Happel-Seidel symmetry $\Nak{15}{4}\simeq\langle 2,4,6\rangle\simeq \Nak{15}{6}$. It follows that $\Nak{14}{4}\simeq \Nak{14}{6}\simeq \langle 2,5,7]$.
     \item[-] For $\X$ of weight type (2,6,6), under Conjecture
     \ref{conj:257 and 266 and 2311(derived version)} we have $\langle 2,6,6]\simeq  \Nak{14}{5}$.
     \item[-] For $\X$ of weight type (2,3,11), under Conjecture
      \ref{conj:257 and 266 and 2311(derived version)} we have $\langle 2,3,11]\simeq  \Nak{16}{8}$.
  \end{itemize}

Secondly, we show that our classification of Nakayama algebras of Fuchsian type above is complete.


In Figure~\ref{figure:green:wall}, the dark green bricks on the right, forming a green wall, contain the Coxeter numbers (see Proposition \ref{Coxeter numbers} and Section \ref{Coxeter polynomial section}). The Fuchsian Nakayama categories, found sofar and listed in Figure~\ref{figure:green:wall}, are fenced by the green wall, and the Nakayama categories belonging to the wall have periodic Coxeter transformations.

The periodicity features of the green wall imply that the members of the green wall cannot be piecewise hereditary, since periodicity only appears for Dynkin module type, which cannot appear here, and for sheaf type where the possible Coxeter numbers --- for weight triples --- are 3, 4, and 6.
On the other hand, the members of the green wall cannot be of Fuchsian type. In fact, by \cite[Table 2]{Lenzing:Pena:2011} there are only 22 weight triples yielding a Fuchsian singularity category with periodic Coxeter transformation. For these 22 triples it is easily checked that they don't belong to the green wall. Therefore, assuming the validity of the perpendicularity conjecture $(F2)$ from Conjecture~\ref{conj:Fuchsian:perp}, the fencing argument from the piecewise hereditary case extends to the Fuchsian case as well. That is, under the assumption $(F2)$, our classification of Nakayama algebras of Fuchsian type is complete.



This finishes the proof.
\end{proof}

\begin{remark} The Coxeter polynomials for extended canonical algebras (in particular, for the Nakayama algebras of Fuchsian type)
have been calculated in
\cite{Lenzing:Pena:2009}. As a consequence of Theorem \ref{classification-of-Fuchsian-nakayama-cat}, we know that Conjecture \ref{conj:Coxeter polynomial and derived equivalence} holds true between Nakayama algebras of Fuchsian type.
\end{remark}

\noindent{\bf Acknowledgment.}
This work was partially supported by the National Natural Science Foundation of China (No. 11801473), and by the Alexander von Humboldt Foundation in the framework of the Alexander von Humboldt Professorship endowed by the German Federal Ministry of Education and Research.

 \appendix
  \renewcommand{\appendixname}{Appendix~\Alph{section}}

\section{Extension-free property}\label{Appendix A}

Let $\mathcal{D}$ be a Hom-finite triangulated category with Serre functor $\bS=\tau[1]$, where $\tau$ is the Auslander-Reiten translation and $[1]$ denotes the suspension functor of $\mathcal{D}$.
The Serre duality of $\mathcal{D}$ is given by the following natural isomorphism $${\rm{Hom}}_{\mathcal{D}}(X,Y)\cong D {\rm{Hom}}_{\mathcal{D}}(Y, \bS X),$$ which are functorial in $X,Y\in\mathcal{D}$.

\begin{lemma}\label{extension-free formula for tau^k U[k]}
Let $X\in\mathcal{D}$ and $U=\bigoplus\limits_{k=0}^{m}\bS^{k}X$ for some $m\geq 0$. Then $U$ is extension-free in $\mathcal{D}$ if and only if ${\rm{Hom}}_{\mathcal{D}}(X, \bS^{k}X[n])=0$ for any $1\leq k\leq m+1$ and any non-zero integer $n$.
\end{lemma}

\begin{proof}
By definition, $U$ is extension-free if and only if for any $n\neq 0$ and $0\leq i,j \leq m$, ${\rm{Hom}}_{\mathcal{D}} (\bS^{i}X, \bS^{j}X[n])=0$, or equivalently, ${\rm{Hom}}_{\mathcal{D}} (X, \bS^{k}X[n])=0$ for any $n\neq 0$ and $-m\leq k \leq m$. For $-m\leq k\leq 0$, we have ${\rm{Hom}}_{\mathcal{D}} (X, \bS^{k}X[n])\cong D{\rm{Hom}}_{\mathcal{D}} (X, \bS^{1-k}X[-n])$ by using  Serre duality, where $1\leq 1-k\leq m+1$. This finishes the proof.
\end{proof}

\begin{lemma}\label{extension-free formula for tau^k U[k]: general case}
Let $X,Y\in\mathcal{D}$ and $U=\bigoplus\limits_{k=0}^{m}\bS^{k}X$, $V=\bigoplus\limits_{k=0}^{r}\bS^{k}Y$ for some $m,r\geq 0$. Assume both of $U$ and $V$ are extension-free. Then $U\oplus V$ is extension-free if and only if ${\rm{Hom}}_{\mathcal{D}}(X, \bS^{k}Y[n])=0$ for any $-m\leq k\leq r+1$ and any non-zero integer $n$.
\end{lemma}

\begin{proof}
By definition, $U\oplus V$ is extension-free if and only if for any $n\neq 0$, $0\leq i \leq m$ and $0\leq j \leq r$, ${\rm{Hom}}_{\mathcal{D}} (\bS^{i}X, \bS^{j}Y[n])=0$ and ${\rm{Hom}}_{\mathcal{D}} (\bS^{j}Y[n], \bS^{i}X)=D{\rm{Hom}}_{\mathcal{D}} (\bS^{i}X, \bS^{j+1}Y[n])=0$.
Equivalently, ${\rm{Hom}}_{\mathcal{D}} (X, \bS^{k}Y[n])=0$ for any $n\neq 0$ and $-m\leq k \leq r+1$. This proves the result.
\end{proof}

In the following, we will show that the objects $T_{(2,4,5)}$, $T_{(2,4,7)}$, $T_{(2,5,5)}$, $T_{(2,5,6)}$ constructed in Theorem \ref{tiltingobjects} are extension-free in $\zvect\X$ case by case.

Recall that for any weight type $(p_1,p_2,p_3)$, there is a surjective group homomorphism $\delta\colon \bl(p_1,p_2,p_3)\rightarrow \mathbb{Z}$ given by $\delta(\vec{x}_i)=\frac{p}{p_i}$ for $1\leq i\leq 3$, where $p={\rm l.c.m.}(p_1,p_2,p_3)$. We denote the $\delta$-datum $(\delta(\vc); \delta(\vx_1), \delta(\vx_2), \delta(\vx_3); \delta(\vw))$ by $\delta(\vc; \vx_1, \vx_2, \vx_3; \vw)$ for convenience.

For a weighted projective line $\X$ of negative Euler characteristic, the associated Fuchsian singularity $R$ is the restricted subalgebra $S|_{\mathbb{Z}\vw}$ of the coordinate algebra of $\X$ defined in \eqref{Coordinate algebra S}, i.e., $$R=\bigoplus_{n\geq0}\Hom{}{\co_{\X}}{ \tau^n\co_{\X}}=S|_{\mathbb{Z}\vw}.$$

\subsection{The weight type (2,4,5)}\label{subsection (245)}

Assume $\X$ has weight type (2,4,5).
By easy calculation we get the $\delta$-datum $\delta(\vc; \vx_1, \vx_2, \vx_3; \vw)=(20;10,5,4;1)$.

According to \cite[Table 3]{Lenzing:Pena:2011}, the semigroup $\{n\vw | n\vw\geq 0\}$ of $\Z \vw$ has minimal generating system $\{4\vw=\vx_3; 10\vw=2\vx_2; 15\vw=\vx_1+\vx_2\}$.
Hence, the subalgebra $R=S|_{\mathbb{Z}\vw}\cong k[x,y,z]/(f)$, where $x=x_3, y=x_2^2, z=x_1x_2$ and $f=z^2+y^3+x^5y$. Moreover, $R$ is $\ZZ$-graded in the sense that $\deg (x,y,z;f)=(4,10,15;30)$.
Thus the second suspension functor of $\zvect\X(2,4,5)$ is given by degree shift: $[2]=(30\vw)=(\vc+2\vx_2)$.


Since $4\vw=\vx_3>0$, the set $\Ss$ defined in \eqref{SS} has the form $$\Ss
=\{\vx\, |\, 0\leq \vx\leq n\vw+\vc,  \text{\ for\ any \ } 2\leq n\leq 5\}
=\{0, \vx_2\}.$$ Hence each line bundle belongs to the orbit $\tau^{\Z}\co$ or $\tau^{\Z}\co(\vx_2)$ by Proposition \ref{rep for line bundle orbits}.

By Proposition \ref{projective cover}, the projective cover of $\co(\vx_2)$, under $\tau^{\ZZ}\co$-exact structure on $\vect\X$, is given by $\cp(\co(\vx_2))=\co\oplus\co(-5\vw)$, which fits into the following exact sequence:
\begin{equation}\label{proj cover for 245}\xymatrix{0\ar[r]&\co(-\vx_1)\ar[rr]^-{(x_1, -x_2)}&&\co\oplus\co(-5\vw)\ar[rr]^-{(x_2, x_1)^t}&& \co(\vx_2)\ar[r]& 0.}
\end{equation}

\begin{lemma} In $\zvect\X(2,4,5)$, we have
$\co(\vx_2)[1]=\co(\vx_2)(15\vw).$
\end{lemma}
\begin{proof}
By \eqref{proj cover for 245} we have $\co(\vx_2)[-1]=\co(-\vx_1)=\co(\vx_2)(-15\vw)$. Then the result follows from $[2]=(30\vw)$.
\end{proof}

Moreover, by Proposition \ref{non-zero morphisms between line bundles}, we have

\begin{lemma} In $\zvect\X(2,4,5)$, we have
\begin{equation}\label{non-zero morphisms(2,4,5)}
  \sHom(\co(\vx_2-\vx), \co(\vx_2))\neq 0 \text{\quad if\ and\ only\ if\quad } 0\leq \vx\leq 4\vx_3.
\end{equation}
\end{lemma}

\begin{proposition}\label{extension-free 245}
$T_{(2,4,5)}$ is extension-free in $\zvect\X(2,4,5)$.
\end{proposition}

\begin{proof}
        Recall that $T_{(2,4,5)}=\bigoplus\limits_{k=0}^{10}L(k\vx_3)$, where $L=\co(\vx_2)$. For construction we assume $T_{(2,4,5)}$ is not extension-free, then there exist some integers $0\leq a, b \leq 10$ and $0\neq m\in \Z$, such that $$\sHom(L(a\vx_{3}), L(b\vx_{3})[m])\neq 0.$$
        Note that $\vx_3=4\vw$ and $L[1]=(L[-1])[2]=L(15\vw)$. According to
        (\ref{non-zero morphisms(2,4,5)}), we have
        \begin{equation}\label{equation5.1} (b-a)\vx_{3}+15m\vw=k\vx_{3} \text{\quad for\  some\ } 0\leq k\leq 4.
        \end{equation}
        It follows that $4(b-a)+15m=4k$ since $\vx_3=4\vw$. Hence $4|m$, say, $m=4r$ for some integer $r\neq0$. Then we have $0\leq k=b-a+15r\leq 4$, a contradiction to the assumption $0\leq a, b \leq 10$. We are done.
\end{proof}

\subsection{The weight type (2,4,7)}\label{subsection (247)}

Assume $\X$ has weight type (2,4,7).
Then the $\delta$-datum is given by $\delta(\vc; \vx_1, \vx_2, \vx_3; \vw)=(28;14,7,4;3)$.

According to \cite[Table 3]{Lenzing:Pena:2011}, the semigroup $\{n\vw | n\vw\geq 0\}$ of $\Z \vw$ has minimal generating system $\{4\vw=3\vx_3; 6\vw=2\vx_2+\vx_3; 7\vw=\vx_1+\vx_2\}$.
Hence, the subalgebra $R=S|_{\mathbb{Z}\vw}\cong k[x,y,z]/(f)$, where $x=x_3^4, y=x_2^2x^3, z=x_1x_2$ and $f=y^3+x^3y+xz^2$. Moreover, $R$ is $\ZZ$-graded in the sense that $\deg (x,y,z;f)=(4,6,7;18)$.
Thus the second suspension functor of $\zvect\X(2,4,7)$ is given by degree shift: $[2]=(18\vw)=(\vc+2\vx_2+3\vx_3)$.


Since $4\vw=3\vx_3>0$, the set $\Ss$ defined in \eqref{SS} has the form $$\Ss
=\{\vx\, |\, 0\leq \vx\leq n\vw+\vc,  \text{\ for\ any \ } 2\leq n\leq 5\}
=\{\vx\,|\, 0\leq \vx\leq \vx_2+2\vx_3\}.$$
Hence each line bundle belongs to the orbit $\tau^{\Z}\co(\vx)$ for some $0\leq \vx\leq \vx_2+2\vx_3$.


By Proposition \ref{projective cover}, the projective covers of $\co(\vx_2)$ and $\co(\vx_2+\vx_3)$, under $\tau^{\ZZ}\co$-exact structure on $\vect\X$, are given by $\cp(\co(\vx_2))=\co\oplus\co(-5\vw)$ and $\cp(\co(\vx_2+\vx_3))=\co\oplus\co(-\vw)$, which fit into the following exact sequences:
\begin{equation}\label{proj cover 1 for 247}\xymatrix{0\ar[r]&\co(-\vx_1-2\vx_3)\ar[rr]^-{(x_1x_3^2, -x_2)}&&\co\oplus\co(-5\vw)\ar[rr]^-{(x_2, x_1x_3^2)^t}&& \co(\vx_2)\ar[r]& 0;}\end{equation}
\begin{equation}\label{proj cover 2 for 247}\xymatrix{0\ar[r]&\co(-\vx_1)\ar[rr]^-{(x_1, -x_2x_3)}&&\co\oplus\co(-\vw)\ar[rr]^-{(x_2x_3, x_1)^t}&& \co(\vx_2+\vx_3)\ar[r]& 0.}
\end{equation}

\begin{lemma} The following statements hold in $\zvect\X(2,4,7)$.
\begin{itemize}
  \item[(1)] $\co(\vx_2)[1]=\co(\vx_2+\vx_3)(7\vw);$
  \item[(2)] $\co(\vx_2+\vx_3)[1]=\co(\vx_2)(11\vw).$
\end{itemize}
\end{lemma}
\begin{proof}
Since we have exact sequences \eqref{proj cover 1 for 247} and \eqref{proj cover 2 for 247} for projective covers, the following hold in $\zvect\X(2,4,7)$:
\begin{itemize}
  \item[-] $\co(\vx_2)[-1]=\co(-\vx_1-2\vx_3)=\co(\vx_2+\vx_3)(-11\vw);$
  \item[-] $\co(\vx_2+\vx_3)[-1]=\co(-\vx_1)=\co(\vx_2)(-7\vw)$.
\end{itemize}
Then the result follows by noting $[2]=(18\vw)$ in $\zvect\X(2,4,7)$.
\end{proof}

Moreover, by Proposition \ref{non-zero morphisms between line bundles}, we have
\begin{lemma}\label{non-zero morphisms (2,4,7)}
For any $\vx=\sum_{1\leq i\leq 3}l_i\vx_i+l\vc$ of normal form, the following hold in $\zvect\X(2,4,7)$.
\begin{itemize}
  \item[(1)] $\sHom(\co(\vx_2-\vx), \co(\vx_2))\neq 0$ if and only if $0\leq \vx\leq \vc+\vx_3$  and   $l_2=0$;
  \item[(2)]  $\sHom(\co(\vx_2+\vx_3-\vx), \co(\vx_2+\vx_3))\neq 0$ if and only if $0\leq \vx\leq \vc$ and $l_1=0$.
      \end{itemize}
\end{lemma}


\begin{proposition}\label{extension-free 247} $T_{(2,4,7)}$ is extension-free in $\zvect\X(2,4,7)$.
\end{proposition}

\begin{proof}
        Recall that $T_{(2,4,7)}=\big(\bigoplus\limits_{k=0}^{6}L(3k\vx_3)\big)\oplus\big
        (\bigoplus\limits_{k=0}^{5}L((3k+1)\vx_3)\big)$, where $L=\co(\vx_2)$. For contradiction we assume $T_{(2,4,7)}$ is not extension-free, then there exist some integers $0\leq a,b\leq 18$ satisfying $a,b\neq 3k+2$ for $0\leq k\leq 5$, and $0\neq m\in \Z$, such that $$\sHom(L(a\vx_{3}), L(b\vx_{3})[m])\neq 0.$$
        Recall that $[2]=(18\vw)$ in $\zvect\X(2,4,7)$. We consider the following two cases.

      \emph{Case 1:} $m$ is even, say, $m=2n$. We get
        $$\sHom(L(a\vx_{3}), L(b\vx_{3})[m])=\sHom(L(a\vx_{3}-18n\vw), L(b\vx_{3})).$$
        For $b=3k, 0\leq k\leq 6$, then according to Lemma \ref{non-zero morphisms (2,4,7)} (1), we have $0\leq (b-a)\vx_{3}+18n\vw \leq \vc+\vx_{3}$ and the coefficient of $\vx_2$ in the normal form of $(b-a)\vx_{3}+18n\vw$ is zero. Hence $4|18n$, i.e, $n$ is even, say, $n=2n'$. Then $(b-a)\vx_{3}+18n\vw=(b-a+27n')\vx_3$. Thus we get $0\leq b-a+27n'\leq 8$. Since $0\leq a, b\leq 18$, we get $n'=0$. It follows that $n=0$ and then $m=0$, a contradiction.

        For $b=3k+1, 0\leq k\leq 5$, then according to Lemma \ref{non-zero morphisms (2,4,7)} (2), we have $0\leq (b-a)\vx_{3}+18n\vw \leq \vc$. If $n$ is even, say, $n=2n'$, then $0\leq (b-a)\vx_{3}+27n'\vx_3 \leq \vc$, hence $0\leq b-a+27n'\leq 7$. Since $0\leq a, b\leq 18$, we get $n'=0$. It follows that $n=0$ and then $m=0$, a contradiction.
        If $n$ is odd, say, $n=2n'-1$, then $0\leq (b-a)\vx_{3}+18n\vw=(b-a+27n'-17)\vx_3+2\vx_2\leq\vc$. It follows that $b-a+27n'-17=0$. Hence $a\equiv 2(\rm{mod}\ 3)$, a contradiction.

        \emph{Case 2:} $m$ is odd, say, $m=2n+1$. By Serre duality we get
        $$\sHom(L(a\vx_{3}), L(b\vx_{3})[2n+1])=D\sHom(L(b\vx_{3}+18n\vw-\vw), L(a\vx_{3})).$$
        According to Lemma \ref{non-zero morphisms (2,4,7)}, by similar arguments as above we obtain that in the normal form of $(a-b)\vx_3-18n\vw+\vw$, the coefficient of $\vx_1$ or $\vx_2$ is zero. It follows that $2|18n-1$, a contradiction.

  This finishes the proof.
\end{proof}

\subsection{The weight type (2,5,5)}

 Assume $\X$ has weight type (2,5,5). Then the $\delta$-datum is given by $\delta(\vc; \vx_1, \vx_2, \vx_3; \vw)=(10;5,2,2;1)$.

According to \cite[Table 3]{Lenzing:Pena:2011}, the semigroup $\{n\vw | n\vw\geq 0\}$ of $\Z \vw$ has minimal generating system $\{4\vw=\vx_2+\vx_3; 5\vw=\vx_1\}$, and the subalgebra $R=S|_{\mathbb{Z}\vw}\cong k[x,y,z]/(f)$, where $x=x_2x_3, y=x_1, z=x_2^5$ and $f=z^2+y^2z+x^5$. Moreover, $R$ is $\ZZ$-graded in the sense that $\deg (x,y,z;f)=(4,5,10;20)$.
Thus the second suspension functor of $\zvect\X(2,5,5)$ is given by degree shift: $[2]=(20\vw)=(2\vc)$.


Since $4\vw=\vx_2+\vx_3>0$, the set $\Ss$ defined in \eqref{SS} has the form $$\Ss
=\{\vx\, |\, 0\leq \vx\leq n\vw+\vc,  \text{\ for\ any \ } 2\leq n\leq 5\}
=\{0, \vx_2, \vx_3, 2\vx_2, 2\vx_3\}.$$
Hence each line bundle belongs to the orbit $\tau^{\Z}\co(j\vx_i)$ for some $0\leq j\leq 2\leq i\leq 3$.



By Proposition \ref{projective cover}, the projective covers of $\co(\vx_3)$ and $\co(2\vx_3)$, under $\tau^{\ZZ}\co$-exact structure on $\vect\X$, are given by $\cp(\co(\vx_3))=\co\oplus\co(-6\vw)$ and $\cp(\co(2\vx_3))=\co\oplus\co(-2\vw)$, which fit into the following exact sequences:
\begin{equation}\label{proj cover 1 for 255}\xymatrix{0\ar[r]&\co(-4\vx_2)\ar[rr]^-{(x_2^4, -x_3)}&&\co\oplus\co(-6\vw)\ar[rr]^-{(x_3, x_2^4)^t}&& \co(\vx_3)\ar[r]& 0;}
\end{equation}
\begin{equation}\label{proj cover 2 for 255}\xymatrix{0\ar[r]&\co(-3\vx_2)\ar[rr]^-{(x_2^3, -x_3^2)}&&\co\oplus\co(-2\vw)\ar[rr]^-{(x_3^2, x_2^3)^t}&& \co(2\vx_3)\ar[r]& 0.}\end{equation}

\begin{lemma}\label{shift for 255} The following statements hold in $\zvect\X(2,5,5)$.
\begin{itemize}
  \item[(1)] $\co(\vx_3)[1]=\co(\vx_2)(10\vw);$
  \item[(2)] $\co(2\vx_3)[1]=\co(2\vx_2)(10\vw).$
\end{itemize}
\end{lemma}
\begin{proof}
Since we have exact sequences \eqref{proj cover 1 for 255} and \eqref{proj cover 2 for 255} for projective covers, the following hold in $\zvect\X(2,5,5)$:
\begin{itemize}
  \item[-] $\co(\vx_3)[-1]=\co(-4\vx_2)=\co(\vx_2)(-10\vw);$
  \item[-] $\co(2\vx_3)[-1]=\co(-3\vx_2)=\co(2\vx_2)(-10\vw)$.
\end{itemize}
Then the result follows by noting $[2]=(20\vw)$ in $\zvect\X(2,5,5)$.
\end{proof}

Moreover, by Proposition \ref{non-zero morphisms between line bundles}, we have

\begin{lemma}\label{non-zero morphisms (2,5,5)}
The following statements hold in $\zvect\X(2,5,5)$.
\begin{itemize}
  \item[(1)] $\sHom(\co(\vx_3-\vx), \co(\vx_3))\neq 0$ if and only if $0\leq \vx\leq \vx_1+3\vx_2$;
  \item[(2)] $\sHom(\co(2\vx_3-\vx), \co(2\vx_3))\neq 0$ if and only if $0\leq \vx\leq \vx_1+2\vx_2+\vx_3.$
      \end{itemize}
\end{lemma}

\begin{proposition}\label{extension-free 255}
$T_{(2,5,5)}$ is extension-free in $\zvect\X(2,5,5)$.
\end{proposition}

\begin{proof}
        Recall that $T_{(2,5,5)}=\bigoplus\limits_{k=0}^{2}\bigoplus\limits_{a=1}^{4}\bS^k(\co(a\vx_3))$. For contradiction we assume $T_{(2,5,5)}$ is not extension-free. Then by Lemma
        \ref{extension-free formula for tau^k U[k]: general case}, there exist some integers $1\leq a, b \leq 4$, $0\leq k\leq 3$ and $0\neq m\in \Z$, such that $$\sHom(\co(a\vx_{3}), \bS^k(\co(b\vx_{3})[m]))\neq 0.$$
         Recall that $\vc=10\vw$ and $[2]=(2\vc)=(20\vw)$. Hence,
         by Lemma \ref{shift for 255} we have
                 $$\bS^k(\co(b\vx_{3})[m])=
         \left\{\begin{array}{lll}
         \co(b\vx_2+k\vw)((k+m)\vc), && \text{\;if\;} k+m \text{\;is\; odd};\\
          \co(b\vx_3+k\vw)((k+m)\vc), && \text{\;if\;} k+m \text{\;is\; even}.
             \end{array} \right.$$
  By Lemma \ref{non-zero morphisms (2,5,5)} (and its symmetric version by exchanging $\vx_2$ and $\vx_3$), we have for $i=2,3$,
\begin{equation}\label{four cases of bvx_i}0\leq b\vx_i-a\vx_3+k\vw+(k+m)\vc\leq
         \left\{\begin{array}{lll}
         \vx_1+3\vx_2, && \text{\;if\;} \co(b\vx_i)\in\tau^\mathbb{Z}(\co(\vx_3));\\
          \vx_1+3\vx_3, && \text{\;if\;} \co(b\vx_i)\in\tau^\mathbb{Z}(\co(\vx_2));\\
          \vx_1+2\vx_2+\vx_3, && \text{\;if\;} \co(b\vx_i)\in\tau^\mathbb{Z}(\co(2\vx_3));\\
          \vx_1+\vx_2+2\vx_3, && \text{\;if\;} \co(b\vx_i)\in\tau^\mathbb{Z}(\co(2\vx_2)).\\
             \end{array} \right.
             \end{equation}
       It follows that $0\leq 2(b-a)+k+10(k+m)\leq 11$ by considering the degrees, which yields that $k+m=0$ or 1.

        \emph{Case 1:}  $k+m=0$, then $2\leq k\leq 3$ since $m\neq 0$ and $(b-a)\vx_3+k\vw\geq 0$, 
        which follows that $(b-a)\vx_3\geq 2\vx_3$, i.e., $b-a\geq 2$. Hence $b=3$ or $4$. If $b=3$, then $a=1$. Observe that $3\vx_3=2\vx_2+2\vw$, i.e., $\co(3\vx_3)\in\tau^\mathbb{Z}(\co(2\vx_2))$. Thus $2\vx_3+k\vw\leq \vx_1+\vx_2+2\vx_3$ by \eqref{four cases of bvx_i}, a contradiction to $2\leq k\leq 3$. If $b=4$, then $a=1$ or $2$. Observe that $4\vx_3=\vx_2+6\vw$. By \eqref{four cases of bvx_i} we have $0\leq (4-a)\vx_3+k\vw\leq \vx_1+3\vx_3$, contradicting to $2\leq k\leq 3$.

        \emph{Case 2:}  $k+m=1$, then $m\neq 0$ implies $k\neq 1$.
         Moreover, $0\leq 2(b-a)+k+10\leq 11$ implies that $b\leq a$.
             \begin{itemize}
             \item[-] If $b=1$, then by \eqref{four cases of bvx_i} we have $0\leq \vx_2-a\vx_3+k\vw+\vc\leq \vx_1+3\vx_3$. By considering the normal forms, we see that the coefficient of $\vx_2$ for the middle term equals zero, which implies $k=1$, a contradiction.
             \item[-] If $b=2$, then by \eqref{four cases of bvx_i} we have $0\leq 2\vx_2-a\vx_3+k\vw+\vc\leq \vx_1+\vx_2+2\vx_3$. By considering the coefficient of $\vx_2$ in the normal form for the middle term, we get $k=2$. In this case, $2\vx_2-a\vx_3+2\vw+\vc=(8-a)\vx_3\leq \vx_1+\vx_2+2\vx_3$, a contradiction.
             \item[-] If $b=3$, then $a=3$ or $4$. Since $3\vx_2=2\vx_3+2\vw$, by
             \eqref{four cases of bvx_i} we have $0\leq 3\vx_2-a\vx_3+k\vw+\vc\leq \vx_1+2\vx_2+\vx_3$. By considering the coefficient of $\vx_3$ in the normal form for the middle term, we get $k=2$. In this case, $3\vx_2-a\vx_3+2\vw+\vc=\vx_2+(8-a)\vx_3\leq \vx_1+2\vx_2+\vx_3$, a contradiction.
             \item[-] If $b=4$, then $a=4$ and $k=0$ since $0\leq 2(b-a)+k+10\leq 11$. In this case, we get $4\vx_2-4\vx_3 +\vc\leq \vx_1+2\vx_2+\vx_3$, a contradiction.
              \end{itemize}
This finishes the proof.
\end{proof}
\subsection{The weight type (2,5,6)}

 Assume $\X$ has weight type (2,5,6). Then the $\delta$-datum is given by $\delta(\vc; \vx_1, \vx_2, \vx_3; \vw)=(30;15,6,5;4)$.

 According to \cite[Table 3]{Lenzing:Pena:2011}, $\{n\vw | n\vw\geq 0\}$ of $\Z \vw$ has minimal generating system $\{4\vw=\vx_2+2\vx_3; 5\vw=\vx_1+\vx_3; 6\vw=4\vx_2\}$, and the subalgebra $R=S|_{\mathbb{Z}\vw}\cong k[x,y,z]/(f)$, where $x=x_2x_3^2, y=x_1x_3, z=x_2^4$ and $f=xz^2+y^2z+x^4$. Moreover, $R$ is $\ZZ$-graded in the sense that $\deg (x,y,z;f)=(4,5,6;16)$.
Thus the second suspension functor of $\zvect\X(2,5,6)$ is given by degree shift: $[2]=(16\vw)=(\vc+4\vx_2+2\vx_3)$.


Since $4\vw=\vx_2+2\vx_3>0$, the set $\Ss$ defined in \eqref{SS} has the form $$\Ss
=\{\vx\, |\, 0\leq \vx\leq n\vw+\vc,  \text{\ for\ any \ } 2\leq n\leq 5\}
=\{\vx\, |\, 0\leq \vx\leq 3\vx_3 \text{\ or } 0\leq \vx\leq 2\vx_2+\vx_3 \}.$$
Hence each line bundle belongs to the orbit $\tau^{\Z}\co(\vx)$ for some $0\leq \vx\leq 3\vx_3$ or $0\leq \vx\leq 2\vx_2+\vx_3$.

%


By Proposition \ref{projective cover}, the projective covers of $\co(j\vx_i)$ for $1\leq j\leq 2\leq i\leq 3$ are given by the middle terms of the following exact sequences respectively:
\begin{equation}\label{proj cover 1 for 256}\xymatrix{0\ar[r]&\co(-\vc)\ar[rr]^-{(x_1^2, -x_2)}&&\co\oplus\co(-6\vw)\ar[rr]^-{(x_2, x_1^2)^t}&& \co(\vx_2)\ar[r]& 0;}
\end{equation}
\begin{equation}\label{proj cover 2 for 256}\xymatrix{0\ar[r]&\co(-\vx_1-3\vx_2)\ar[rr]^-{(x_1x_2^3, -x_3)}&&\co\oplus\co(-7\vw)\ar[rr]^-{(x_3, x_1x_2^3)^t}&& \co(\vx_3)\ar[r]& 0;}
\end{equation}
\begin{equation}\label{proj cover 3 for 256}\xymatrix{0\ar[r]&\co(-4\vx_3)\ar[rr]^-{(x_3^4, -x_2^2)}&&\co\oplus\co(-2\vw)\ar[rr]^-{(x_2^2, x_3^4)^t}&& \co(2\vx_2)\ar[r]& 0;}
\end{equation}
\begin{equation}\label{proj cover 4 for 256}\xymatrix{0\ar[r]&\co(-3\vx_2)\ar[rr]^-{(x_2^3, -x_3^2)}&&\co\oplus\co(-2\vw)\ar[rr]^-{(x_3^2, x_2^3)^t}&& \co(2\vx_3)\ar[r]& 0.}
\end{equation}

\begin{lemma} The following statements hold in $\zvect\X(2,5,6)$.
\begin{itemize}
  \item[(1)] $\co(\vx_2)[1]=\co(2\vx_3)(6\vw);$
  \item[(2)] $\co(\vx_3)[1]=\co(\vx_2+\vx_3)(5\vw);$
  \item[(3)] $\co(2\vx_2)[1]=\co(2\vx_2)(8\vw);$
  \item[(4)] $\co(2\vx_3)[1]=\co(\vx_2)(10\vw).$
\end{itemize}
\end{lemma}
\begin{proof}
Since we have exact sequences \eqref{proj cover 1 for 256}--\eqref{proj cover 4 for 256} for projective covers, the following hold in $\zvect\X(2,5,6)$:
\begin{itemize}
  \item[-] $\co(\vx_2)[-1]=\co(-\vc)=\co(2\vx_3)(-10\vw);$
  \item[-] $\co(\vx_3)[-1]=\co(-\vx_1-3\vx_2)=\co(\vx_2+\vx_3)(-11\vw);$
  \item[-] $\co(2\vx_2)[-1]=\co(-4\vx_3)=\co(2\vx_2)(-8\vw);$
  \item[-] $\co(2\vx_3)[-1]=\co(-3\vx_2)=\co(\vx_2)(-6\vw).$
\end{itemize}
Then the result follows by noting $[2]=(16\vw)$ in $\zvect\X(2,5,6)$.
\end{proof}
%
%
Moreover, by Proposition \ref{non-zero morphisms between line bundles}, we have

\begin{lemma}\label{Hom relations for 256}
For any $\vx=\sum_{1\leq i\leq 3}l_i\vx_i+l\vc$ of normal form, the following hold in $\zvect\X(2,5,6)$.
\begin{itemize}
  \item[(1)] $\sHom(\co(\vx_2-\vx), \co(\vx_2))\neq 0$ if and only if $0\leq \vx\leq \vx_1+5\vx_3$;
  \item[(2)] $\sHom(\co(\vx_3-\vx), \co(\vx_3))\neq 0$ if and only if $0\leq \vx\leq \vc+2\vx_2 \text{\ and\ } l_3=0$;
  \item[(3)] $\sHom(\co(2\vx_2-\vx), \co(2\vx_2))\neq 0$ if and only if $0\leq \vx\leq \vx_1+\vx_2+3\vx_3;$
  \item[(4)] $\sHom(\co(2\vx_3-\vx), \co(2\vx_3))\neq 0$ if and only if $0\leq \vx\leq \vx_1+2\vx_2+\vx_3.$
\end{itemize}
\end{lemma}

\begin{proposition}\label{extension-free 256}
$T_{(2,5,6)}$ is extension-free in $\zvect\X(2,5,6)$.\end{proposition}

\begin{proof}
        Recall that $T_{(2,5,6)}=\big(\bigoplus\limits_{k=0}^{3}\bS^k(\co(\vx_3))\big)\oplus \big( \bigoplus\limits_{k=0}^{2}\bS^k(\co(2\vx_3)\oplus \co(4\vx_3)\oplus \co(6\vx_3))\big).$ In $\zvect\X(2,5,6)$ we have $[2]=(16\vw)$, $\co(4\vx_3)=\co(2\vx_2)(2\vw)$ and $\co(6\vx_3)=\co(\vx_2)(6\vw)$.
        For any objects $X,Y\in\zvect\X(2,5,6)$, by Serre duality we have $$\sHom(X, Y[2n+1])=D\sHom(Y[2n], \tau X)=D\sHom(Y, X((1-16n)\vw)).$$
        Denote by $V_0=\bigoplus\limits_{k=0}^{3}\bS^{k}(\co(\vx_3))$ and $V_i=\bigoplus\limits_{k=0}^{2}\bS^{k}(\co(2i\vx_3))$ for $1\leq i\leq 3$. Then $T=V_0\oplus V_1\oplus V_2\oplus V_3$. We consider the following steps.
        \begin{itemize}
        \item[(i)] $V_0$ is extension-free. For contradiction,
            by Lemma \ref{extension-free formula for tau^k U[k]} there exist $m\neq 0$ and $0\leq k\leq 4$, such that $\sHom(\co(\vx_3), \tau^{k}\co(\vx_3)[k+m])\neq 0$.
            By Lemma \ref{Hom relations for 256} we have
            \begin{itemize}
                \item[-]$\sHom(\co(\vx_3), \co(\vx_3+n\vw))\neq 0$ if and only if $n=0$ or $6$.
            \end{itemize}
            If $k+m=2n$ for some $n$, then $k+16n=0$ or 6, it follows that $k=n=0$; if $k+m=2n+1$ for some $n$, then $1-(k+16n)=0$ or 6, it follows that $n=0$ and $k=1$. In both cases we get $m=0$, a contradiction.
        \item[(ii)] $V_i$ is extension-free for any $1\leq i\leq 3$. For contradiction, by Lemma \ref{extension-free formula for tau^k U[k]} there exist $m\neq 0$ and $0\leq k\leq 3$, such that $\sHom(\co(2i\vx_3), \tau^{k}\co(2i\vx_3)[k+m])\neq 0$. By Lemma \ref{Hom relations for 256} we have
            \begin{itemize}
                \item[-]$\sHom(\co(2\vx_3), \co(2\vx_3+n\vw))\neq 0$ if and only if $n=0$ or 5;
                \item[-]$\sHom(\co(4\vx_3), \co(4\vx_3+n\vw))\neq 0$ if and only if $n=0,4,5$ or 9;
                \item[-]$\sHom(\co(6\vx_3), \co(6\vx_3+n\vw))\neq 0$ if and only if $n=0$ or 5.
            \end{itemize}
            If $k+m=2n$ for some $n$, then $k+16n\in\{0,4,5,9\}$, it follows that $k=n=0$; if $k+m=2n+1$ for some $n$, then $1-(k+16n)\in\{0,4,5,9\}$, it follows that $n=0, k=1$. In both cases we get $m=0$, a contradiction.
        \item[(iii)] $V_0\oplus V_i$ is extension-free for any $1\leq i\leq 3$. For contradiction, by Lemma \ref{extension-free formula for tau^k U[k]: general case} there exist $m\neq 0$ and $-3\leq k\leq 3$, such that $\sHom(\co(\vx_3), \tau^{k}\co(2i\vx_3)[k+m])\neq 0$.
            By Lemma \ref{Hom relations for 256} we have
            \begin{itemize}
                \item[-]$\sHom(\co(\vx_3), \co(2i\vx_3+n\vw))\neq 0$ if and only if $n=0$;
                \item[-]$\sHom(\co(2i\vx_3), \co(\vx_3+n\vw))\neq 0$ if and only if $n=4i+1$.
            \end{itemize}
            If $k+m=2n$ for some $n$, then $k+16n=0$, it follows that $k=n=0$, then $m=0$, a contradiction; if $k+m=2n+1$ for some $n$, then $1-(k+16n)=4i+1$, it follows that $n=0$, then $k=-4i$, a contradiction.
        \item[(iv)] $V_i\oplus V_j$ is extension-free for any $1\leq i<j\leq 3$. For contradiction, by Lemma \ref{extension-free formula for tau^k U[k]: general case} there exist $m\neq 0$ and $-2\leq k\leq 3$, such that $\sHom(\co(2i\vx_3), \tau^{k}\co(2j\vx_3)[k+m])\neq 0$.
            By Lemma \ref{Hom relations for 256} we have
            \begin{itemize}
                \item[-]$\sHom(\co(2i\vx_3), \co(2j\vx_3+n\vw))\neq 0$ if and only if $n=0$ or $5$;
                \item[-]$\sHom(\co(2j\vx_3), \co(2i\vx_3+n\vw))\neq 0$ if and only if $n=4(j-i)$ or $4(j-i)+5$.
            \end{itemize}
            If $k+m=2n$ for some $n$, then $k+16n=0$ or $5$, it follows that $k=n=0$, then $m=0$, a contradiction; if $k+m=2n+1$ for some $n$, then $1-(k+16n)=4(j-i)$ or $4(j-i)+5$, it follows that $n=0$ and $k=1-4(j-i)$ or $-4-4(j-i)$, a contradiction.
        \end{itemize}
Then the proof is finished.
\end{proof}





\def\cprime{$'$} \def\cprime{$'$} \def\cprime{$'$}

\bibliographystyle{plain}
\bibliography{reference}

\end{document}